\documentclass[leqno,11pt]{amsart}
\usepackage[top=1.25in, bottom=1.25in, left=1in, right=1in]{geometry}
\usepackage{amssymb, amsmath,amsmath,latexsym,amssymb,amsfonts,amsbsy, amsthm,esint}
\usepackage{color}
\usepackage{graphicx}
\definecolor{pku}{RGB}{139,0,18}
\usepackage{enumerate}
\usepackage{hyperref}

\let\pa=\partial

\let\b=\beta

\let\d=\delta

\let\lam=\lambda
\let\r=\rho

\let\f=\frac
\let\tl=\tilde
\let\p=\psi
\let\om=\omega
\let\G= \Gamma
\let\D=\Delta
\let\Lam=\Lambda

\let\Om=\Omega

\let\e=\varepsilon
\let\pa=\partial

\let\ri=\rightarrow
\let\na=\nabla

\def\di{\mathrm{div}\,}

\newcommand{\beq}{\begin{equation}}
\newcommand{\eeq}{\end{equation}}
\newcommand{\beqo}{\begin{equation*}}
\newcommand{\eeqo}{\end{equation*}}
\newcommand{\ben}{\begin{eqnarray}}
\newcommand{\een}{\end{eqnarray}}
\newcommand{\beno}{\begin{eqnarray*}}
\newcommand{\eeno}{\end{eqnarray*}}

\newtheorem{lem}{Lemma}[section]

\newtheorem{theorem}{Theorem}[section]

\newtheorem{lemma}[theorem]{Lemma}
\newtheorem{proposition}[theorem]{Proposition}

\theoremstyle{remark}
\newtheorem{step}{Step}
\newtheorem{case}{Case}
\newtheorem{rmk}{Remark}[section]

\newcommand{\mm}{\mathbf{m}}

\newcommand{\BR}{\mathbb{R}}

\newcommand{\BQ}{\mathcal{Q}}

\newcommand{\tf}{\tilde{F}}

\newcommand{\pq}{\partial\BQ_{phy}}

\begin{document}

\title[Tensor-valued Variational Obstacle Problem]{Regularity of Minimizers of a Tensor-valued Variational Obstacle Problem in Three Dimensions}

\author{Zhiyuan Geng}
\address{Courant Institute, New York University\\
251 Mercer Street, New York, NY 10012, USA}
\email{zhiyuan@cims.nyu.edu}

\author{Jiajun Tong}
\address{University of California, Los Angeles\\
Box 951555, Los Angeles, CA 90095, USA}
\email{jiajun@math.ucla.edu}

\date{\today}

\begin{abstract}
Motivated by Ball and Majumdar's modification of Landau-de Gennes model for nematic liquid crystals, we study energy-minimizer $Q$ of a tensor-valued variational obstacle problem in a bounded 3-D domain with prescribed boundary data. The energy functional is designed to blow up as $Q$ approaches the obstacle. Under certain assumptions, especially on blow-up profile of the singular bulk potential, we prove higher interior regularity of $Q$, and show that the contact set of $Q$ is either empty, or small with characterization of its Hausdorff dimension.
We also prove boundary partial regularity of the energy-minimizer.
\end{abstract}

\maketitle

\section{Introduction}
\label{section: intro}

\subsection{Background and problem formulation}
In this paper, we consider a tensor-valued variational obstacle problem originates from the Landau-de Gennes model for nematic liquid crystals. We shall study the regularity of the minimizer and estimate the size of the contact set.
\par
In the Landau-de Gennes theory \cite{DP,MN}, local state of nematic liquid crystals at spatial point $x\in\mathbb{R}^3$ is characterized by a $3\times 3$-tensor-valued order parameter $Q(x)$ in
\beqo
\BQ=\{Q\in \mathbb{R}^{3\times 3}:\,Q_{ij}=Q_{ji},\, Q_{ii}=0\},
\eeqo
interpreted as traceless second moment of (formal) probability distribution function $f$ of local molecular orientation over $\mathbb{S}^2$, i.e.,
\begin{equation*}
Q(x)=\int_{\mathbb{S}^2}\left(\mm\otimes\mm-\frac{1}{3}I\right)f(x,\mm)\,d\mm.
\end{equation*}
Here $I$ denotes the $3\times 3$-identity matrix.
It follows that all eigenvalues of $Q$ should belong to $[-\f{1}{3},\f{2}{3}]$.
If one of the eigenvalues is equal to $-1/3$ or $2/3$, formally the density $f(x,\mm)\,d\mm$ is concentrated on a measure-zero subset of $\mathbb{S}^2$, in which case the bulk energy density will be infinity \cite{BM}.
We introduce the set of admissible configuration of nematic liquid crystal
\begin{equation}\label{constraint}
Q\in \BQ_{phy}:=\left\{Q\in\BQ:\,\lam_1(Q), \lam_2(Q), \lam_3(Q)\in \left[-\f{1}{3},\f{2}{3}\right]\right\},
\end{equation}
where $\lambda_i(Q)$ $(i=1,2,3)$ are three eigenvalues of $Q$.
$\BQ_{phy}$ is a bounded close subset of $\BQ$ with
\ben
\label{boundary of Q}\pa\BQ_{phy}&=&\left\{Q\in \BQ:\, \lam_i(Q)\in \left[-\f13,\f23\right] \text{ for } i=1,2,3;\text{ at least one of } \lam_i(Q)=-\f13 \right\}.\\
\label{interior of Q} \BQ^{\mathrm{o}}_{phy} &=&\left\{ Q\in \BQ:\, \lam_i(Q)\in \left(-\f13,\f23\right) \text{ for } i=1,2,3 \right\}
\een

Here we used the fact that $\sum Q_{ii} = 0$.
For more background on Q-tensor models, readers are referred to \cite{DP} and \cite{MN}.

Motivated by Ball-Majumdar model of nematic liquid crystals \cite{BM,KKLS},
we introduce
\beq\label{energy}
E[Q]=\int_{\Om}f_e(\nabla Q(x))+f_b(Q(x))\,dx.
\eeq
Here $\Om$ is a bounded smooth domain in 3-D, $f_e(\nabla Q)$ represents elastic energy density and $f_b(Q)$ is a singular bulk energy potential.

We assume
\beq\label{elastic}
f_{e}(\nabla Q)=\f12|\nabla Q|^2+\f{A}{2}|\di Q|^2 = \f12 Q_{ij,k} Q_{ij,k} + \f{A}{2}Q_{ij,j}Q_{ik,k}.
\eeq
Here
$Q_{ij,k}$ is short for $\pa_{x_k} Q_{ij}$, and summation is taken over repeated indices.

In fact, \eqref{elastic} is reduced from a more general form of elastic energy
\beq\label{general elastic}
f_e(\nabla Q)=\f{L_1}{2}|\nabla Q|^2+\f{L_2}{2} |\di Q|^2+\f{L_3}{2}Q_{ij,k}Q_{ik,j},
\eeq
where $L_1$, $L_2$, and $L_3$ are elastic constants,
satisfying that \cite{LMT}
\beq\label{relation}
L_1+L_3>0,\quad 2L_1-L_3>0,\quad L_1+\f53 L_2+\f16 L_3>0.
\eeq
Under these conditions, $f_e$ is coercive, i.e., there exist positive $c<C$ such that
\beqo %\label{coercive}
c|\nabla Q|^2\leq f_e(\nabla Q)\leq C|\nabla Q|^2.
\eeqo
It is well-known that with proper coefficients, the last two terms in \eqref{general elastic} can be combined into a null Lagrangian (see e.g., \cite[Lemma 1.1]{Zoc}), which can be omitted under strong anchoring (Dirichlet) boundary conditions (see below).
Hence, without loss of generality, we may simplify \eqref{general elastic} into \eqref{elastic} by setting $L_3 = 0$ and $L_1=1$.
Thanks to \eqref{relation}, when $A>-3/5$,  $f_e$ in \eqref{elastic} is coercive and strictly convex in $\nabla Q$.

The following assumptions on $f_b$ will be made throughout this work.
\begin{enumerate}[($i$)]
  \item \label{assumption for f_b: smooth and convex}In $\BQ^\mathrm{o}_{phy}$, $1\leq f_b(Q)<+\infty$, and it is convex and smooth.
      The lower bound of $f_b$ is not an essential constraint, since the constant $1$ is irrelevant in minimizing \eqref{energy} when $\Om$ is fixed.
  \item \label{assumption for f_b: infinity}$f_b(Q)=\infty$ in $\mathcal{Q}\backslash\BQ^{\mathrm{o}}_{phy}$, and $f_b(Q)\ri +\infty$ as $Q\ri\pa\BQ_{phy}$ \cite{BM}.

\end{enumerate}
More hypotheses on $f_b$ will be proposed later.

To this end, we shall minimize \eqref{energy} over the set of admissible configurations given by
\beq\label{admissible}
\mathcal{A}=\{Q(x)\in H^1(\Om,\BQ_{phy}):\,Q=Q_0 \text{ on }\pa \Om\}
\eeq
with some $Q_0\in C^\infty(\overline{\Om},\BQ^{\mathrm{o}}_{phy})$.
Here the Dirichlet boundary condition is understood in the trace sense.
We take $Q_0$ to be smooth for simplicity, which may be weaken.
It is also noteworthy that $Q_0$ is separated from the obstacle $\pa \BQ_{phy}$ throughout $\overline{\Om}$.
Since $Q_0\in \mathcal{A}$ with $E(Q_0)<+\infty$ and $\mathcal{A}$ is closed, one can easily prove existence of energy minimizer of \eqref{energy} in $\mathcal{A}$ by the direct method of calculus of variations.
Moreover, the minimizer is unique thanks to the convexity of $f_e$ and $f_b$.
We still denote it as $Q$ with abuse of notation.

In spite of existence and uniqueness, it is not clear if $Q$ will touch the obstacle $\pa\BQ_{phy}$ in $\Om$.
For this purpose, we introduce \emph{contact set} $\mathcal{C}$ of $Q$,
\beqo %\label{def of contact set}
\mathcal{C}:=\{x\in \Omega: \, Q(x)\in\pq\}.
\eeqo
Since $E[Q]<+\infty$, by the assumption \eqref{assumption for f_b: infinity} on $f_b$, $\mathcal{C}$ has Lebesgue measure zero.
However, one would naturally expect that $\mathcal{C} = \varnothing$ ---
this will make more sense as we establish regularity of $Q$ below since right now $Q$ is only an $H^1$-function but not well-defined pointwise.
In this work, we shall propose a number of conditions to pursue emptiness of $\mathcal{C}$.
In the cases where $\mathcal{C}=\varnothing$ cannot be guaranteed, we provide estimates for Hausdorff dimension of $\mathcal{C}$.

Let us review some related works before stating our main results. As is mentioned earlier, the model we are studying originates from the work by Ball and Majumdar \cite{BM}, where the authors proposed a thermotropic bulk potential that blows up as the eigenvalue of $Q$ approaches $-\f13$. They also proved in the one-elastic-constant case (i.e., $A = 0$) that, under weak assumptions on $f_b$ (see \cite[Theorems 4 and 6, Corollaries 3 and 4]{BP}), the contact set is empty and $Q$ is smooth in the domain. Bauman and Phillips \cite{BP} studied the 2-D case of this problem with a Landau-de Gennes elastic energy density containing more elasticity terms.
Using a hole-filling technique, they proved H\"{o}lder continuity of the energy minimizer $Q$ in general, and showed that if some of the elastic coefficients are zero, one should find $Q$ is smooth.
Evans, Kneuss and
Tran \cite{Evans2016partial} established a partial regularity result for a general form of elastic energy, stating that $Q$ is smooth outside a zero-measure closed subset of $\Om$.
They also proved that under certain hypotheses on $f_b$, the singular set (contact set) has Hausdorff dimension at most $(n-2)$ in the $n$-dimensional case. Readers are referred to \cite{feireisl2014evolution, feireisl2015nonisothermal,wilkinson2015strictly, iyer2015dynamic} and references therein for results on dynamic problems with singular bulk potential modeling liquid crystal flows.

\subsection{Main results}

Our main strategy in this work is to derive the emptiness of $\mathcal{C}$ from its regularity.
We start with a key result concerning interior regularity of $Q$.
\begin{proposition}[Interior regularity of $Q$]
\label{prop: higher interior regularity}
Assume $A>-3/5$ and $f_b$ is convex.
Let $Q$ denote the unique minimizer of $E[Q]$ in $\mathcal{A}$.
Then
\begin{enumerate}
  \item $Q\in H^2_{loc}(\Om)$.
  \item $\nabla Q\in L^q_{loc}(\Om)$, for $q$ such that
  \beq\label{eqn: condition on q}
  1\leq q <6p(A),
  \eeq
  where
  \beq\label{eqn: def of p(A)}
  p(A) =
  \begin{cases}
  1+\f{3}{A}+\f{9}{5A^2},&\mbox{ if }A\in \left(-\f35,\sqrt{\f{18}{5}}\right],\\
  1+\f{3+5A}{2\sqrt{10}A- 6},&\mbox{ if }A\in \left(\sqrt{\f{18}{5}},\f35+\sqrt{\f{18}{5}}\right],\\
  1+\f{3+\sqrt{9+6A}}{2A},&\mbox{ if }A>\f35+\sqrt{\f{18}{5}}.
  \end{cases}
  \eeq
  In particular, $p(A)$ is decreasing in $A$ when $A>0$, and $p(A)>1$.
\end{enumerate}
\end{proposition}

\begin{rmk}
  The first part of the Proposition is a special case of Theorem 4.1 in \cite{Evans2016partial}, while the second part is stronger than the $W^{1,6}$-regularity that directly follows from the first part by 3-D Sobolev embedding.
\end{rmk}

Proposition \ref{prop: higher interior regularity} shows that the minimizer $Q(x)$ can be realized as a H\"{o}lder continuous function in $\Omega$ and it is thus well-defined \emph{pointwise}.
Therefore, the contact set $\mathcal{C}$ is now well-defined.

As is mentioned above, the energy-minimizer in
the special case $A = 0$ has been studied in \cite{BM,BP}, where $f_e(\na Q)$ reduces to the Dirichlet energy and the uniform separation from the obstacle simply follows from the maximum principle. Then smoothness of $Q$ can be justified by elliptic regularity theory \cite{DG} applied to the Euler-Lagrange equation.
When $A\neq 0$, this technique fails.
Nevertheless, it is tempting to believe when $A$ is close to $0$, the minimizer should behave like the one for $A = 0$.
This leads to the following theorem with a perturbative nature.

\begin{theorem}[Small $A$ case]\label{thm: small A results}
Assume $f_b$ satisfies the assumptions \eqref{assumption for f_b: smooth and convex} and \eqref{assumption for f_b: infinity}.
For any $V\subset\subset \Om$,
there exists $A_*>0$ depending on $V$, such that for any $|A|\leq A_*$, the minimizer $Q_A\in \mathcal{A}$ of \eqref{energy} with parameter $A$ satisfies $\mathcal{C}\cap V = \varnothing$.
Moreover, $Q_A$ is smooth in $V$.
\end{theorem}

For general $A>-\f35$, it is natural to believe that whether $\mathcal{C} =\varnothing$ should depend on the growth of $f_b$ near $\pa\BQ_{phy}$.
Suppose $\mathcal{C}\not = \varnothing$ and $x_0\in \mathcal{C}$.
If $f_b$ grows fast near $\pa\BQ_{phy}$, in a neighborhood of $x_0$, $Q(x)$ must leave $\pa\BQ_{phy}$ sufficiently fast in order to not incur huge bulk energy.
However, this may lead to large elastic energy, which prohibits $Q$ from being energy-minimizing.
In this spirit, we propose the following hypothesis on the growth of $f_b$ near $\pa \BQ_{phy}$:

\begin{enumerate}[($i$)]
\setcounter{enumi}{2}
\item \label{assumption for f_b: growth}
    We assume there exists some $s>0$ and $0<m_s<M_s$, such that in $\BQ_{phy}$,
    \beq\label{bulk}
    \f{m_s}{d(Q)^s}\leq f_b(Q)\leq \f{M_s}{d(Q)^s},
    \eeq
    or there exists $0<k_0\leq K_0$ and $m_0\leq M_0$, such that for any $Q\in \pa \mathcal{Q}_{phy}$, there exists $k(Q)\in [k_0,K_0]$ satisfying
    \beq
    k(Q)|\ln (d(\lam Q))|+m_0\leq  f_b(\lam Q)\leq k(Q)|\ln (d(\lam Q))|+M_0, \; \forall\, \lam\in(0,1).
    \label{eqn: log growth non-uniform}%\\
%    \label{eqn: bound for k(Q)} &k_0\leq k(Q)\leq K_0
    \eeq
    Here $d(Q) := \mathrm{dist}(Q,\BQ_{phy}^c)$ is the distance measured in the Frobenius norm.
\end{enumerate}
It can be shown that the Ball-Majumdar potential satisfies the assumption \eqref{eqn: log growth non-uniform}.
See \cite[(22)-(25)]{BM}.
Then we have
\begin{theorem}%[A sufficient condition for $\mathcal{C} = \varnothing$]
\label{thm1}
Assume $A>-3/5$ and $f_b$ satisfies the assumptions \eqref{assumption for f_b: smooth and convex}, \eqref{assumption for f_b: infinity}, and \eqref{assumption for f_b: growth} with \eqref{bulk}. %, and \eqref{assumption for f_b: gradient bound}.
Let $Q$ be the unique minimizer of $E[Q]$ in $\mathcal{A}$.

Then
\begin{enumerate}
  \item \label{part: empty contact set} If $s>s(A)$, where
  \beq\label{eqn: condition on s}
%s(A) :=\frac{10A^2}{5A^2+30A+18},
s(A) :=\frac{2}{2p(A)-1},
\eeq
we have $\mathcal{C} = \varnothing$ and $Q$ is smooth in $\Om$. %\in C^{1,\alpha}_{loc}(\Om)$ for all $\alpha\in(0,1)$.
Here $p(A)$ is defined in \eqref{eqn: def of p(A)}.

\item \label{part: Hausdorff dimension of singular set}
Otherwise, if $0<s\leq s(A)$, %for all $V\subset\subset \Om$,
$$
\dim_{H}\mathcal{C} \leq 3\left(1-\frac{s}{s(A)}\right).
$$
Here $\dim_H\mathcal{C}$ denotes the Hausdorff dimension of $\mathcal{C}$.
\end{enumerate}
\end{theorem}

\begin{rmk}
We do not claim optimality of the borderline exponent $s(A)$, but we stress that $s(A)<2$ for all $A>-\f35$ since $p(A)>1$.
As a result, Theorem \ref{thm1}, under slightly stronger assumptions on the energy functional, significantly improves the result of \cite[Theorem 4.2]{Evans2016partial} in 3-D case, where it was shown that $\dim_H \mathcal{C}\leq 1$ in 3-D provided that $f_b$ grows like $d(Q)^{-2}$ near $\pa \BQ_{phy}$.
\end{rmk}

\begin{rmk}
It is noteworthy that $s(A)\ri 0$ as $A\ri 0$.
This echoes with the results that when $A = 0$, only very weak blow-up of $f_b$ near $\pa \BQ_{phy}$ is needed to achieve emptiness of $\mathcal{C}$\cite{BM,BP}.
\end{rmk}

In \cite{Evans2016partial}, inspired by a model of the Ball-Majumdar singular bulk potential, where $f_b$ grows like $|\ln d(Q)|$ near $\pa \BQ_{phy}$, an additional assumption is proposed on $f_b$.
It can be roughly stated as follows in our context: for some $c_0>0$,
\beq
\f{\pa^2 f_b}{\pa Q_{ij} \pa Q_{mn}}y_{ij}y_{mn}\geq c_0\left|\f{\pa f_b}{\pa Q}\cdot y\right|^2\quad\mbox{ for all }y\in T\BQ_{phy}.
\label{eqn: stronger assumptions by EKT log case}
\eeq
Based on this, the authors proved that $\dim_H \mathcal{C} \leq 1$ (see \cite[Theorem 4.3]{Evans2016partial}).
This assumption has been verified rigorously in \cite{feireisl2015nonisothermal} for the Ball-Majumdar potential \cite{BM}.
In what follows, we shall propose similar hypotheses on $f_b$ and improve our results in Theorem \ref{thm1}.
\begin{enumerate}[($i$)]
\setcounter{enumi}{3}
  \item \label{assumption for f_b: gradient bound} If $f_b$ satisfies the assumption \eqref{assumption for f_b: growth} with \eqref{bulk}, we assume that for some $C_s,c_s >0$,
      \beq
      |Df_b(Q)|^s\leq C_s f_b(Q)^{s+1}\quad\mbox{ for all }Q\in \BQ^{\mathrm{o}}_{phy},
      \label{eqn: gradient bound inverse power case}
      \eeq
      and
      \beq
      \f{\pa^2 f_b}{\pa Q_{ij} \pa Q_{mn}}y_{ij}y_{mn}\geq \f{c_s}{f_b}\left|\f{\pa f_b}{\pa Q}\cdot y\right|^2\quad\mbox{ for all }y\in T\BQ_{phy}.
      \label{eqn: stronger assumptions inverse power case}
      \eeq
      If, otherwise, $f_b$ satisfies the assumption \eqref{assumption for f_b: growth} with \eqref{bulk log}, we assume that for some $C_0 >0$,
      \beq
      |Df_b(Q)|\leq C_0 \exp(k_0^{-1}f_b(Q))\quad\mbox{ for all }Q\in \BQ^{\mathrm{o}}_{phy},
      \label{eqn: gradient bound log case}
      \eeq
      and also \eqref{eqn: stronger assumptions by EKT log case} holds.
      %Here $k_0$ is given in \eqref{bulk log}.
\end{enumerate}

\begin{rmk}
Note that \eqref{eqn: gradient bound inverse power case} and \eqref{eqn: stronger assumptions inverse power case} are inspired by assuming that $f_b(Q)$ behaves like $d(Q)^{-s}$ near $\pa \BQ_{phy}$;
while \eqref{eqn: gradient bound log case} is derived by assuming $f_b(Q)$ behaves like $k_0|\ln d(Q)|$ near $\pa \BQ_{phy}$.
We omit their derivations.
\end{rmk}

With the new assumption \eqref{assumption for f_b: gradient bound}, we may improve the results in part \eqref{part: Hausdorff dimension of singular set} of Theorem \ref{thm1}.
\begin{theorem}\label{thm: improved dimension bound}
Assume $A>-3/5$.
Let $Q$ be the unique minimizer of $E[Q]$ in $\mathcal{A}$.
\begin{enumerate}
  \item \label{part: inverse power case}
  Suppose $f_b$ satisfies the assumptions \eqref{assumption for f_b: smooth and convex}, \eqref{assumption for f_b: infinity}, \eqref{assumption for f_b: growth} with \eqref{bulk} and \eqref{assumption for f_b: gradient bound}.
  When $0<s\leq s(A)$,
  \beq
  \dim_H\mathcal{C}\leq
  \begin{cases}
  3(1-\f{3s+2}{3s(A)+2}),&\mbox{ if }p(A)\leq 2,\\
  3-\f{s}{s(A)}-\f{2(2+s)}{2+s(A)},&\mbox{ if }p(A)> 2.
  \end{cases}
  \label{eqn: improved dimension bound inverse power case}
  \eeq

  \item \label{part: log case}
    Suppose $f_b$ satisfies the assumptions \eqref{assumption for f_b: smooth and convex}, \eqref{assumption for f_b: infinity}, \eqref{assumption for f_b: growth} with \eqref{eqn: log growth non-uniform} and \eqref{assumption for f_b: gradient bound}.
  Then
  \beq
  \dim_H\mathcal{C}\leq
  \begin{cases}
  3-\f{6p(A)-3}{6p(A)\cdot \f{K_0-k_0}{k_0}+2p(A)+2},&\mbox{ if }p(A)\leq 2,\\
  3-\f{2p(A)-1}{p(A)\left(\f{K_0-k_0}{k_0}p(A)+1\right)},&\mbox{ if }p(A)> 2.
  \end{cases}
  \label{eqn: improved dimension bound log case}
  \eeq
\end{enumerate}
\end{theorem}
Unfortunately, because of the limitation of our approach, Theorem \ref{thm: improved dimension bound} is not an improvement of Theorem 4.3 in \cite{Evans2016partial}.

Since no continuity of $Q$ up to the boundary has be established, we are unable to extend the above discussion (even in the small $A$ case) from the interior up to $\partial \Omega$.
Instead, we prove the following partial regularity result at $\partial\Om$.

\begin{theorem}[Boundary partial regularity]\label{thm: boundary partial regularity}
Assume $A > -3/5$ and $f_b$ satisfies the assumptions \eqref{assumption for f_b: smooth and convex} and \eqref{assumption for f_b: infinity}.
The minimizer $Q\in \mathcal{A}$ of \eqref{energy} is H\"{o}lder continuous in $\Omega$ up to $\partial\Om\backslash S$ for some $S\subset \partial\Omega$, with $\mathcal{H}^1(S) = 0$.
Here $\mathcal{H}^1(\cdot )$ denotes the 1-dimensional Hausdorff measure.

In particular, $\mathcal{H}^1(\overline{\mathcal{C}}\cap \partial\Om) = 0$.
\end{theorem}

The rest of the paper is organized as follows.
Section \ref{section: higher interior regularity} is devoted to proving Proposition \ref{prop: higher interior regularity} on the interior regularity of the energy-minimizer and Theorem \ref{thm: small A results}.
Theorem \ref{thm1} and Theorem \ref{thm: improved dimension bound}
will be proved in Section \ref{section: proof of main theorem}.
We show Theorem \ref{thm: boundary partial regularity} on the boundary partial regularity in Section \ref{section: boundary partial regularity}.
In the Appendices, we derive the formula for $p(A)$ in Proposition \ref{prop: higher interior regularity} in Appendix \ref{section: formula for p(A)}.
Useful properties of the distance function $d(Q)$ are proved in Appendix \ref{section: properties of d(Q)}.
Finally, in Appendix \ref{section: construction of f_b eps}, we shall present a construction of an approximating sequence of $f_b$ that will be used in Section \ref{section: proof of the improved bound}.

\section{Interior Regularity of the Minimizer and Proof of Theorem \ref{thm: small A results}}\label{section: higher interior regularity}

We first show Proposition \ref{prop: higher interior regularity}.
As is mentioned before, the $H^2_{loc}$-regularity has been established in \cite[Theorem 4.1]{Evans2016partial} under weaker assumptions.
The proof there uses standard arguments in the calculus of variations \cite{Evans}, with special care of the singular bulk energy.
For completeness, we still present it here in our context.
Then we shall generalize its idea to prove the $W^{1,q}_{loc}$-regularity.

\begin{proof}[Proof of Proposition \ref{prop: higher interior regularity}]
\setcounter{step}{0}
\begin{step}[Basic setup]
Fix an open subset $V$ of $\Om$ such that $V\subset\subset \Om$.
We select another open subset $W$ such that $V\subset\subset W\subset\subset \Om$.
Take a nonnegative smooth cutoff function $\xi$ in $\BR^3$ supported on $\overline{W}$, such that $
\xi\equiv 1$ on $V$ and $\|\nabla \xi\|_{L^\infty(\Om)}\leq C(V,W)$.
Let $e_1,e_2,e_3$ be the standard coordinate vectors in $\BR^3$.
With $p\geq 1$ is to be determined, let
\beq\label{v1}
u_k= D_k^{-h}(\xi^2|D^h_k Q|^{2p-2}D_k^hQ),
\eeq
where
$$
D_k^hQ(x):=\f{Q(x+he_k)-Q(x)}{h}. %,\quad |D^h Q|^2 =\sum_{m = 1}^3|D_m^h Q|^2.
$$
Here in order to make the above two quantities well-defined throughout $\Omega$, we make zero extension of $Q$ outside $\Omega$ (still denoted by $Q$).

Consider
\begin{equation}
\begin{split}
&\;(Q+\e u_k)(x) \\
= &\;\f{\e(\xi^2|D_k^hQ|^{2p-2})(x-he_k)}{h^2}Q(x-h e_k)+\f{\e(\xi^2|D_k^hQ|^{2p-2})(x)}{h^2}Q(x+he_k)\\
&\;+\left(1-\f{\e(\xi^2|D_k^hQ|^{2p-2})(x-he_k)+\e (\xi^2|D_k^hQ|^{2p-2})(x)}{h^2}\right)Q(x).
\end{split}
\label{eqn: the candidate}
\end{equation}
It is easy to show that $(Q+\e u_k) = Q = Q_0$ on $\partial\Omega$ if $h\leq C(W,\Om)$.

\end{step}

\begin{step}[Interior $H^2$-regularity]
Let $p = 1$ in \eqref{eqn: the candidate}.
Provided that $0<\e\ll h^2$, $(Q+\e u_k)(x)$ is a convex combination of $Q(x)$, $Q(x+he_k)$, and $Q(x-he_k)$, %, provided that $0<\e\ll h^{2p}$.
which implies that $(Q+\e u_k)(x)\in \mathcal{Q}_{phy}$ for $\forall\,x\in \Omega$.
On the other hand, it is obvious that $(Q+\e u_k)$ has $H^1$-regularity and satisfies the boundary condition.
Hence, $(Q+\e u_k)\in\mathcal{A}$.
By the convexity of $f_b$,
\begin{equation*}
\begin{split}
&\;f_b(Q+\e u_k)\\
%=&\;f_b\left(\f{\e\xi(x-he_k)^2}{h^2}Q(x-he_k)+\f{\e\xi(x)^2}{h^2}Q(x+he_k)+\left(1-\f{\e\xi(x-he_k)^2}{h^2}-\f{\e\xi(x)^2}{h^2}\right)Q(x)\right)\\
\leq &\; \f{\e\xi(x-he_k)^2}{h^2}f_b(Q(x-he_k))+\f{\e\xi(x)^2}{h^2}f_b(Q(x+he_k))+\left(1-\f{\e\xi(x-he_k)^2}{h^2}-\f{\e\xi(x)^2}{h^2}\right)f_b(Q(x)).
\end{split}
\end{equation*}
Hence,
\beqo
\begin{split}
&\;\int_{\Om} f_b(Q+\e u_k)-f_b(Q)\,dx\\
\leq &\; \f{\e}{h^2}\int_{\BR^3} \xi(x-he_k)^2[f_b(Q(x-he_k))-f_b(Q(x))]+\xi(x)^2[f_b(Q(x+he_k))-f_b(Q(x))]\,dx\\
=&\;0.
\end{split}
\eeqo
In the last line, we used change of variables.
Since $Q$ minimizes $E[Q]$ in $\mathcal{A}$, i.e.,
\beqo
\int_{\Om}f_e(\nabla(Q+\e u_k))+f_b(Q+\e u_k)\,dx\geq \int_{\Om} f_e(\nabla Q)+f_b(Q)\,dx,
\eeqo
we deduce that
\beq\label{eqn: elastic energy comparison}
\int_\Om f_e(\nabla Q)\,dx\leq \int_\Om f_e(\nabla(Q+\e u_k))\,dx.
\eeq
Sending $\e\ri 0^+$ yields a variational inequality,
$$
\int_{\Om }\f{\pa f_e(\nabla Q)} {\pa Q_{ij,l}}\cdot \pa_l\left( D^{-h}_k(\xi^2 D_k^h Q_{ij})\right)\,dx\geq 0.
$$
Here the summation convention applies to all repeated indices except for $k$;
it will always be this case in the rest of this proof.
Using integration by parts for $D_k^{-h}$,
$$
\int_{\Om}\xi^2 D_k^h\left(\f{\pa f_e(\nabla Q)}{\pa Q_{ij,l}}\right)  D_k^h Q_{ij,l}\,dx
\leq -\int_{\Om}D_k^h\left(\f{\pa f_e(\nabla Q)}{\pa Q_{ij,l}}\right) \cdot 2\xi\pa_l \xi D_k^h Q_{ij}\,dx.
$$
Recall that $f_e(\nabla Q)$ is quadratic in $\nabla Q$, and it is coercive.
Hence, with $a>0$ to be determined,
$$%\beq\label{H2est1}
c\int_{\Om} \xi^2 |D_k^h \nabla Q|^2\,dx\leq a\int_{\Om} \xi^2 |D_k^h \nabla Q|^2\,dx +\f{C}{a}\int_{\Om} |\nabla \xi|^2|D_k^h Q|^2\,dx,
$$%\eeq
where $c$ and $C$ are universal constants only depending on $A$.
Taking $a<c/2$, we end up having
$$
\int_\Om \xi^2 |D_k^h \nabla Q|^2\,dx \leq C \int_{W} |D_k^h Q|^2\,dx\leq C \int_{\Om} |\nabla Q|^2\,dx,
$$
where $h\ll 1$ and $C>0$ only depends on $A$ and $\xi$.
This bound is independent of $h$, which implies that \cite{Evans}
\beqo
\int_V |\nabla^2 Q|^2\,dx\leq C(A,V)\int_{\Om} |\nabla Q|^2\,dx.
\eeqo
Hence, $Q\in H^2_{loc}(\Om)$.
\end{step}

\begin{step}[Interior $W^{1,q}$-regularity]
Take $p>1$ to be determined.
By Sobolev embedding, in 3-D, $Q\in H^2_{loc}(\Omega)$ implies $Q$ is locally bounded.
Hence, in the interior of $\Om$, if $h\ll 1$, $|D_k^h Q|\leq C/h$.
In fact, in our problem, $Q$ should enjoy a trivial $L^\infty$-estimate since $\BQ_{phy}$ is bounded.
However, we avoid using this natural $L^\infty$-bound, so that the same proof can be applied to situations where no $L^\infty$-estimate is a priori available.
This point will be useful later in Section \ref{section: proof of the improved bound}.
Hence, in the interior of $\Om$, $(Q+\e u_k)(x)$ is a convex combination of $Q(x)$, $Q(x+he_k)$, and $Q(x-he_k)$, provided that $0<\e\ll h^{2p}$.
This implies $(Q+\e u_k)(x)\in \mathcal{Q}_{phy}$ for $\forall\,x\in \Om$.
On the other hand, combining the interior $H^2$-regularity of $Q$ with \eqref{eqn: the candidate}, $(Q+\e u_k)\in H^1(\Omega, \BQ_{phy})$ for arbitrary $p> 1$.
Therefore, $(Q+\e u_k)\in \mathcal{A}$.
Then we argue as before by the convexity of $f_b$ to find that \eqref{eqn: elastic energy comparison} still holds, which allows us to derive a similar variational inequality,
\beqo
\int_{\Om}D_k^h\left(\f{\pa f_e(\nabla Q)}{\pa Q_{ij,l}}\right) \pa_l (\xi^2 |D_k^h Q|^{2p-2}D_k^h Q_{ij})\,dx\leq 0.
\eeqo
Substituting the form of $f_e$ in \eqref{elastic} into this inequality, we calculate that
\begin{equation*}
\begin{split}
&\;\int_{\Om} D_k^h(Q_{ij,l}+\delta_{jl}A(\di Q)_i)\left[D_k^hQ_{ij,l}|D_k^h Q|^{2p-2}\xi^2\right.\\
&\;\left.\qquad+ \xi^2 D_k^h Q_{ij} \cdot(2p-2)|D_k^h Q|^{2p-4}(D_k^h\pa_l Q : D_k^h Q)\right]\,dx \\
\leq &\;-\int_\Om  D_k^h(Q_{ij,l}+\delta_{jl}A(\di Q)_i)\cdot 2\xi\pa_l\xi|D_k^h Q|^{2p-2} D_k^h Q_{ij}\,dx.
\end{split}
\end{equation*}
Here $D_k^h\partial_l Q: D_k^hQ := D_k^h Q_{ij,l}D_k^hQ_{ij}$.
By Cauchy-Schwarz inequality, with $c>0$ to be determined,
\begin{equation}\label{imp1}
\begin{split}
&\;\int_\Om \xi^2(|D_k^h \nabla Q|^2+A|D_k^h\di Q|^2)|D_k^h Q|^{2p-2}\,dx\\
&\;+\int_\Om (2p-2)\xi^2 |D_k^h\nabla Q: D_k^hQ|^2|D_k^hQ|^{2p-4}\,dx\\
&\;+\int_\Om A(2p-2)\xi^2|D_k^hQ|^{2p-4}(D_k^h\di Q)_i \cdot (D_k^hQ)_{ij}\cdot (D_k^h\nabla Q: D_k^hQ)_j\,dx\\
\leq &\;c \int_\Om \xi^2|D_k^h \nabla Q|^2|D_k^h Q|^{2p-2}\,dx+ \f{C}{c}\int_{\Om} |\nabla \xi|^2 |D_k^hQ|^{2p}\,dx,
\end{split}
\end{equation}
where $C>0$ is a universal constant only depending on $A$.
Introducing another parameter $\om \in [0,1]$ to be determined, which satisfies
\beq
\f{3}{5}\om+A \geq 0,
\label{eqn: constraint on omega}
\eeq
we derive that
\begin{equation*}
\begin{split}
&\;(\om |D_k^h \nabla Q|^2+A|D_k^h\di Q|^2)|D_k^h Q|^{2p-2}\\
\geq &\;\left(\f35\om+A\right)|D_k^h\di Q|^2|D_k^h Q|^{2p-2}\\
\geq &\;\frac{3}{2}\left(\f35\om+A\right)|D_k^h\di Q|^2\|D_k^h Q\|_2^2|D_k^h Q|^{2p-4},
\end{split}
\end{equation*}
where $\|\cdot\|_2$ denotes the matrix 2-norm.
Here we used
\beq\label{keyinequality}
\|D_k^h Q\|_2^2 \leq \frac{2}{3}|D_k^h Q|^2,\quad |D_k^h\di Q|^2\leq \f53|D_k^h \nabla Q|^2.
\eeq
Indeed, they can be justified by using the fact that $D_k^h Q$ and $D_k^h \pa_l Q$ are symmetric traceless $3\times 3$-matrices.  We give a detailed proof of these two inequalities in Appendix \ref{section: proof of key inequality}.

On the other hand, %it is trivial that
\beqo
(1-\om)|D_k^h \na Q|^2 |D_k^h Q|^{2p-2}\geq (1-\om) |D_k^h\nabla Q: D_k^hQ|^2|D_k^hQ|^{2p-4}
\eeqo
Hence, by Young's inequality,
\begin{equation*}
\begin{split}
&\;(|D_k^h \nabla Q|^2+A|D_k^h\di Q|^2)|D_k^h Q|^{2p-2}+(2p-2)|D_k^h\nabla Q: D_k^hQ|^2|D_k^hQ|^{2p-4}\\
\geq &\;2\sqrt{\frac{3}{2}\left(\f35\om +A\right)[(1-\om)+(2p-2)]}|D_k^hQ|^{2p-4}|(D_k^h\di Q)_i \cdot (D_k^hQ)_{ij}\cdot (D_k^h\nabla Q: D_k^hQ)_j|.
\end{split}
\end{equation*}
If $p>1$ satisfies that
\beqo
2\sqrt{\frac{3}{2}\left(\f35\om +A\right)[(1-\om)+(2p-2)]}>|A(2p-2)|,
%\label{eqn: constraint on p with omega primitive form}
\eeqo
i.e.,
\beq
1<p<1+\f{\f95\left(\om+\f53 A\right)+\sqrt{\f{81}{25}\left(\om+\f53 A\right)^2+\f{18}{5}A^2\left(\om+\f53 A\right)(1-\om)}}{2A^2}=:p(A,\om),
\label{eqn: constraint on p with omega}
\eeq
then we are able to take $c>0$ suitably small in \eqref{imp1} to obtain that
\beq\label{imp3}
\int_\Om \xi^2|D_k^h \nabla Q|^2|D_k^h Q|^{2p-2}\,dx\leq C\int_\Om |\nabla \xi|^2 |D_k^hQ|^{2p}\,dx,
\eeq
where $C>0$ depends only on $A$, $p$ and $\om$.

By maximizing the right hand side of \eqref{eqn: constraint on p with omega} over all admissible $\om$, we may take any $p\in(1,p(A))$, where
$$
p(A) := \sup_{\om\in [0,1],\,\om+\f{5}{3}A\geq 0}p(A,\om).
$$
We shall show in Lemma \ref{lemma: formula for p(A)} that $p(A)$ has a more explicit form given by \eqref{eqn: def of p(A)}.

To this end, by \eqref{imp3}, with $h\ll 1$,
$$%\beq\label{imp4}
\int_\Om |\nabla(\xi|D_k^h Q|^p)|^2\,dx\leq C\int_{\Om} |\nabla \xi|^2|D_k^h Q|^{2p}\,dx\leq C(\xi,A,p)\int_{W'}|\nabla Q|^{2p}\,dx.
$$%\eeq
Here $W'$ is an open subset of $\Om$, such that $W\subset\subset W'\subset\subset \Om$.
By Sobolev embedding,
$$
\|D_k^h Q\|_{L^{6p}(V)}\leq C(V,W,A,p)\|\nabla Q\|_{L^{2p}(W')}.
$$
This estimate is independent of $h$ as long as $h\ll 1$.
Hence, %Taking $h\rightarrow 0^+$ and summing over all $k$ yields that
$$
\|\nabla Q\|_{L^{6p}(V)}\leq C(V,W,A,p)\|\nabla Q\|_{L^{2p}(W')}.
$$
Combining this with the interior $H^2$-regularity in the previous step, we conclude, by making iterations if needed, that
\beq\label{imp5}
\|\nabla Q\|_{L^{6p}(V)}\leq C(V,A,p)\|\nabla Q\|_{L^2(\Om)},
\eeq
as long as $p\in (1,p(A))$.
This proves the interior $W^{1,q}$-regularity of $Q$, where $q = 6p$.
\end{step}
\end{proof}
Theorem \ref{thm: small A results} then follows from a compactness-type argument.

\begin{proof}[Proof of Theorem \ref{thm: small A results}]
Suppose the statement is false for some $V\subset\subset \Om$.
Then there exists a sequences $A_i\ri 0$, such that the minimizer $Q_{A_i}\in \mathcal{A}$ of \eqref{energy} with parameter $A_i$ admits $d(Q_{A_i}(x_i))=0$ at some $x_i\in V$.
We may assume $x_i \ri x_*\in \overline{V}$.

Take $W$ such that $V\subset\subset W\subset\subset \Om$.
Since the $H^2(W)$-estimate in Proposition \ref{prop: higher interior regularity} is uniform for all $A$ sufficiently close to $0$, $Q_{A_i}$ has uniform bound in $C^{1/2}(W)$.
By Arzel\`{a}-Ascoli lemma, up to a subsequence, there exists $Q_*$ such that $Q_{A_i}\ri Q_*$ uniformly in $\overline{V}$.
On the other hand, the uniform-in-$A$ $H^2_{loc}(\Om)$-estimate also implies that, up to a further sequence, the convergence is strong in $H^1_{loc}(\Om)$.
Without loss of generality, we may assume that $Q_{A_i}\ri Q_*$ and $\nabla Q_{A_i}\ri \nabla Q_*$ almost everywhere in $\Om$.
Hence,
\beqo
f_b(Q_{A_i})\ri f_b(Q_*),\,f_e(\nabla Q_{A_i})\ri f_e(\nabla Q_*),\text{ a.e.\;in }\Om.
\eeqo
Let $Q^*$ denote the energy minimizer when $A = 0$.
By Fatou's Lemma,
\beqo
\int_\Om f_e(\nabla Q_*)+f_b(Q_*)\,dx\leq \lim_{i\to + \infty}E_{A_i}[Q_{A_i}]\leq \lim_{i\to +\infty}E_{A_i}[Q^*] = E_{0}[Q^*].
\eeqo
By the uniqueness of the energy minimizer, $Q_* = Q^*$.
Meanwhile, by the uniform convergence of $Q_{A_i}$ to $Q_*$ in $\overline{V}$, the uniform $C^{1/2}(W)$-bound of $Q_{A_i}$, and the fact that $d(\cdot)$ is Lipschitz continuous (see Lemma \ref{lemma: d(Q) is Lipschitz}),
$$
d(Q_*(x_*)) = \lim_{i\to +\infty}d(Q_{A_i}(x_*)) \leq \lim_{i\to +\infty}d(Q_{A_i}(x_i))+C|x_*-x_i|^{1/2} = 0,
$$
which implies $x_*\in \mathcal{C}$.
This contradicts with the fact that $\mathcal{C} = \varnothing$ when $A = 0$ \cite{BM,BP}.

Since $\mathcal{C} = \varnothing$ in $V$, $Q_A$ satisfies an Euler-Lagrange equation on $V$,
$$
-\D Q_A - A\na \mathrm{div}\,Q_A = -Df_b(Q_A).
$$
With $f_b$ smooth, the smoothness of $Q_A$ follows from the regularity theory of elliptic systems by a bootstrap argument.
\end{proof}

\section{Proof of Theorem \ref{thm1}}%, and Theorem \ref{thm: Hausdorff dimension of contact set}}
\label{section: proof of main theorem}
In this section, we shall prove Theorem \ref{thm1}.
With $a >0$ sufficiently small, we define
\beq
\Om_a = \{x\in\Omega:\; d(Q(x))\leq a\}.
\label{eqn: def of lower level set of distance}
\eeq
The main idea of the proof is to derive a bound for the size of $\Omega_a$.

\begin{proof}[Proof of Theorem \ref{thm1}]
\setcounter{step}{0}
\begin{step}

Let
$\eta_a:\;[0,\infty)\ri\mathbb{R}_+$ be defined by
$$
\eta_a(x) = \min\left\{1,
\frac{1-\sqrt{6}a}{1-\sqrt{6}x}\right\}.
$$
By Lemma \ref{lemma: characterization of d(Q)} (in particular, \eqref{formula for d(Q)}), it can be verified that the map $h_a: Q\mapsto \eta_a(d(Q))Q$ retracts $\BQ_{phy}$ to a smaller subset of $\BQ^{\mathrm{o}}_{phy}$. %which has distance $a$ from $\BQ_{phy}^c$.
To be more precise,
%such that,
$h_a(Q) \equiv Q$ if $d(Q)> a$, and $d(h_a(Q))=a$ if $d(Q)\leq a$.
Moreover, since $\eta_a\circ d$ is piecewise-smooth on $\BQ_{phy}$, by a limiting argument, $h_a(Q)\in H^1(\Om)$.

Take arbitrary $U\subset\subset V\subset\subset \Omega$.
Define a smooth cut-off function $\rho\in C_0^\infty(V)$ such that $\rho\in [0,1]$ and $\rho \equiv 1$ on $U$.
Let
$$
Q_a(x) = \rho h_a(Q) + (1-\rho)Q.
$$
Obviously when $a\ll 1$, $Q_a\in \mathcal{A}$.
Also, by taking $a$ even smaller if needed, we may assume that $f_b(0)\leq f_b(Q)$ for all $Q\in \BQ_{phy}$ such that $d(Q)\leq  a$.
Then by the convexity of $f_b$, $f_b(Q)\geq f_b(Q_a)$ for all $x\in \Omega$, as $Q_a(x)$ can be viewed as a convex combination of $Q(x)$ and $0$.
\end{step}

\begin{step}
We shall use $Q_a$ as a comparison configuration in \eqref{energy}.
By the minimality of $Q$, % in \eqref{energy},
\beq\label{eqn: comparing Q_a and Q}
\int_{\Om}f_b(Q(x))-f_b(Q_a(x))\,dx\leq \int_{\Om}f_e(\nabla Q_a(x))-f_e(\nabla Q(x))\,dx.
\eeq
With $\lambda>1$ to be determined, we derive that
\beqo%\label{eqn: lower bound of LHS preliminary step}
\int_{\Om}f_b(Q(x))-f_b(Q_a(x))\,dx
\geq \int_{\Omega_{a/\lambda}\cap U} m_s \left(\frac{a}{\lambda}\right)^{-s} - M_s a^{-s}\,dx.
\eeqo
Taking $\lam = \Lam_s$, with
\beq
\Lam_s := \left(\f{2M_s}{m_s}\right)^\f{1}{s},
\label{eqn: def of Lambda s}
\eeq
we find that
\beq\label{eqn: lower bound of LHS}
\int_{\Om}f_b(Q(x))-f_b(Q_a(x))\,dx
\geq \frac{m_s}{2}\cdot  \left(\frac{a}{\Lam_s}\right)^{-s}|\Omega_{a/\Lam_s}\cap U|.
\eeq
On the right hand side of \eqref{eqn: comparing Q_a and Q}, since $f_e$ is quadratic,
\beqo
\begin{split}
&\;\left|\int_{\Om}f_e(\nabla Q_a(x))-f_e(\nabla Q(x))\,dx\right|\\
\leq &\; C\int_{\Omega}|\nabla (Q_a-Q)|^2 + |\nabla Q||\nabla (Q_a-Q)|\,dx\\
= &\; C\int_{\Omega_a\cap V}|\nabla (\rho Q(\eta_a(d(Q))-1))|^2 + |\nabla Q||\nabla  (\rho Q(\eta_a(d(Q))-1))|\,dx.
\end{split}
\eeqo
Since $\eta_a\circ d$ is a piecewise-smooth map with Lipschitz constant $1$ and $|\eta_a-1|\leq Ca$,
\beqo
\left|\int_{\Om}f_e(\nabla Q_a(x))-f_e(\nabla Q(x))\,dx\right|
\leq C\int_{\Omega_a\cap V}|\nabla Q|^2+a^2 |\nabla \rho|^2 \,dx.
\eeqo
By H\"older's inequality and Proposition \ref{prop: higher interior regularity},
$$
\|\nabla Q\|_{L^2(\Omega_a\cap V)}\leq \|\nabla Q\|_{L^q(\Omega_a\cap V)}|\Omega_a\cap V|^{\frac{1}{2}-\frac{1}{q}}\leq C\|\nabla Q\|_{L^2(\Omega)}|\Omega_a\cap V|^{\frac{1}{2}-\frac{1}{q}}.
$$
where $q$ satisfies \eqref{eqn: condition on q} and $C$ depends on $q$, $V$ and $A$.
Therefore, for $a\ll 1$,
\beq\label{eqn: upper bound for RHS}
\left|\int_{\Om}f_e(\nabla Q_a(x))-f_e(\nabla Q(x))\,dx\right|
\leq C(|\Omega_a\cap V|^{1-2/q}+ a^2 |\Omega_a\cap V|)\leq C|\Omega_a\cap V|^{1-2/q},
\eeq
where $C$ depends on $A$, $q$, $U$, $V$ and $E[Q]$.
Combining \eqref{eqn: comparing Q_a and Q}, \eqref{eqn: lower bound of LHS}, and \eqref{eqn: upper bound for RHS}, we obtain that
\beq\label{eqn: iterative bounds on the size of up level set}
\frac{|\Omega_{a/\Lam_s}\cap U|}{(a/\Lam_s)^s}\leq C|\Omega_a\cap V|^{1-2/q}.
\eeq
\end{step}

\begin{step}

We first show part \eqref{part: Hausdorff dimension of singular set} of Theorem \ref{thm1}.

Since $|\Om_a\cap V|\leq C$, we find that for all $a\ll 1$, by virtue of \eqref{eqn: iterative bounds on the size of up level set}, $|\Om_{a/\Lam_s}\cap U|\leq C(a/\Lam_s)^s$.
With the notation changed, this is equivalent to that, for all $V\subset\subset \Omega$ and $a\ll 1$, $|\Om_a\cap V|\leq Ca^s$.
Applying this new bound in \eqref{eqn: iterative bounds on the size of up level set} yields that $|\Om_a\cap V|\leq Ca^{s+s(1-2/q)}$, with a different constant $C$.
By repeating this procedure, we can show that for all
\beq
\beta < sq/2,
%\label{eqn: decay exponent of size of upper level set}
\eeq
all $V\subset\subset \Om$,
\beq\label{eqn: bounds for up level set}
|\Omega_a\cap V|\leq Ca^\beta\quad \mbox{ for all } a\ll 1,
\eeq
where $C$ depends on $\beta$, $A$, $M_s$, $m_s$, $q$, $V$, $E[Q]$ and $\Om$, but not on $a$.

Now we assume $\mathcal{C}\cap U\not = \varnothing$ for some $U\subset\subset V\subset\subset\Om$, and $x_0\in \mathcal{C}\cap U$; otherwise, we have nothing to prove.
By Proposition \ref{prop: higher interior regularity} and Sobolev embedding,
$Q\in C^{\alpha}_{loc}(\Om)$ where $\alpha = 1-3/q$.
Hence, for all $x\in B_r(x_0)$ with $r \leq d(x_0,\partial V)/2$, $d(Q(x)) \leq Cr^\alpha$, where $C$ does not depend on $r$.
Therefore, for $r\ll 1$,
\beq\label{eqn: neighborhood of C contained in up level set}
(\mathcal{C}\cap U)_r \subset \{d(Q(x))\leq Cr^{\alpha}\}\cap V, %\{f_b(Q(x))\geq C(V)r^{-\alpha s}\},
\eeq
where $(\mathcal{C}\cap U)_r := \cup_{x\in \mathcal{C}\cap U}B_r(x)$ is the $r$-neighborhood of $\mathcal{C}\cap U$.
Combining this with \eqref{eqn: bounds for up level set}, we find that for all $\gamma <s(q-3)/2$, and $r\ll 1$,
\beq\label{eqn: bounds on the size of neighborhood}
|(\mathcal{C}\cap U)_r|\leq Cr^{\gamma}.
\eeq
Since $\cup_{x\in \mathcal{C}\cap U}\overline{B_r(x)}$ is a covering of $(\mathcal{C}\cap U)_r$ with finite radius bound, by Vitali Covering Lemma, there exist a countable set $\mathcal{C}'_r = \{x_i\}\subset \mathcal{C}\cap U$, such that $\overline{B_r(x_i)}$ are disjoint, and $\cup_{x\in \mathcal{C}'_r}\overline{B_{5r}(x)}$ is a covering of $\mathcal{C}\cap U$.
By the disjointness of $\overline{B_r(x_i)}$, $\cup_{x\in \mathcal{C}'_r}B_r(x)\subset (\mathcal{C}\cap U)_r$, which together with \eqref{eqn: bounds on the size of neighborhood} implies that $\mathcal{C}_r'$ is a finite set with $|\mathcal{C}_r'|\leq Cr^{\gamma-3}$.
Since $\mathcal{C}\cap U \subset \cup_{x\in \mathcal{C}'_r} \overline{B_{5r}(x)}$, we find $\mathcal{H}^{3-\gamma}_{10r}(\mathcal{C}\cap U)\leq C$ (see e.g.\;\cite{lin2002geometric} for the notation), where $C$ is independent of $r$.
Therefore, %the $(3-\gamma)$-dimensional Hausdorff measure of $\mathcal{C}\cap U$ is finite, i.e.,
$\mathcal{H}^{3-\gamma}(\mathcal{C}\cap U)< +\infty$.

Since $\gamma$ is arbitrary as long as $\gamma<s(q-3)/2$ with $q$ satisfying \eqref{eqn: condition on q}, we conclude that
$$
\dim_{H}(\mathcal{C}\cap U) \leq 3-\frac{3s}{s(A)}
$$
for any open subset $U\subset \Om$.
Applying this estimate to an exhaustion $\{U_i\}_{i=1}^\infty$ of $\Om$, we can prove part \eqref{part: Hausdorff dimension of singular set} of Theorem \ref{thm1}.
\end{step}

\begin{step}
Finally, we prove part \eqref{part: empty contact set} of Theorem \ref{thm1} by contradiction.

Suppose $x_0 \in \mathcal{C}$ for some $s>s(A)$.
Take $x_0\in  V\subset\subset W\subset\subset \Om$.
By \eqref{eqn: neighborhood of C contained in up level set}, for all $a\ll 1$,
$B_{Ca^{1/\alpha}}(x_0)\subset \{d(Q)\leq a\}\cap V$, where $\alpha = 1-3/q$ and $q$ satisfies \eqref{eqn: condition on q}.
Hence,
$$
|\Om_a \cap V|\geq |B_{Ca^{1/\alpha}}(x_0)| = Ca^{3/\alpha} = Ca^{3q/(q-3)} =: Ca^{q_0},
$$
where $C$ is independent of $a$.
Applying this bound on the left hand side of \eqref{eqn: iterative bounds on the size of up level set} yields that, for all $a\ll 1$,

$$
|\Om_a \cap W| \geq C a^{\frac{q_0 - s}{1-2/q}}=: Ca^{q_1}.
$$
Here $C$ is a different universal constant.
Repeating this procedure, we obtain a sequence $\{q_n\}$ by
$$
q_n = \frac{q_{n-1}-s}{1-2/q}\quad \mbox{ for }n\geq 1,
$$
such that for any $U \subset\subset \Om$, and $a\ll 1$, $|\Om_a \cap U|\geq Ca^{q_n}$.
Here the constant $C$ should depend on $n$ but not on $a$.
However, if $s >2q_0/q$, it is easy to verify that there exists $N<\infty$, such that $q_N <0$, and
$$
|\Om_a \cap U|\geq Ca^{q_N}\quad \mbox{ for all }a\ll 1,
$$
which is obviously impossible.

By \eqref{eqn: condition on q} and \eqref{eqn: condition on s}, whenever $s>s(A)$, $s>2q_0/q = 6/(q-3)$ is achievable by suitably choosing $q$.
This proves $\mathcal{C} = \varnothing$.
\end{step}
\end{proof}

\section{Proof of Theorem \ref{thm: improved dimension bound}}
\label{section: proof of the improved bound}

Recall that in the proof of Proposition \ref{prop: higher interior regularity}, we got rid of the singular bulk energy term at the very beginning thanks to the convexity of $f_b$.
However, the new assumptions \eqref{eqn: stronger assumptions by EKT log case} and \eqref{eqn: stronger assumptions inverse power case} on $f_b$ allow us to make use of the bulk energy term in a better way, which leads to the improvement in Theorem \ref{thm: improved dimension bound}.

In what follows, we shall first recast the proof of Proposition \ref{prop: higher interior regularity} to derive an estimate involving $f_b(Q)$.
However, as it is hard to directly work with $f_b$ with singularity,
we shall introduce an approximating sequence $\{f_b^\e\}_{0<\e\ll 1}$ of $f_b$ with no singularity in the entire $\mathcal{Q}$.
Indeed, we have

\begin{lemma}
\label{lemma: construction of f_b eps}
Suppose $f_b$ satisfies the assumptions \eqref{assumption for f_b: smooth and convex} and \eqref{assumption for f_b: infinity}.
Then there exists $\{f^\e_b\}_{0<\e\ll 1}$ satisfying the following conditions:
\begin{enumerate}[($i'$)]
\item For all $0<\e\ll 1$, $f_b^\e(Q)\in [0, \infty)$ for all $Q\in \mathcal{Q}$;
\item $f_b^\e$ are convex and smooth in $\mathcal{Q}$;
\item $f_b^\e(Q)\leq f_b(Q)$ for all $Q\in \mathcal{Q}$.
\item Moreover,
$$
\lim_{\e\ri 0^+} f_b^\e(Q) = f_b(Q),\quad \lim_{\e\ri 0^+} Df_b^\e(Q) = Df_b(Q)
$$
locally uniformly in $\BQ^\mathrm{o}_{phy}$.
\end{enumerate}
Here $Df_b^\e(Q)$ denotes the gradient of $f_b^\e$ with respect to $Q$.

\end{lemma}
Its proof will be given in Appendix \ref{section: construction of f_b eps}.

Given $\e>0$, let $Q^\e$ denote the unique minimizer of
\beqo
E^\e[Q]=\int_\Om f_e(\nabla Q)+f_b^\e(Q)\,dx
\eeqo
in $H^1(\Omega, \mathcal{Q})$ subject to the Dirichlet boundary condition \eqref{admissible}.
Indeed, existence and uniqueness of $Q^\e$ can be justified by classic arguments. We omit the details.

Concerning $Q^\e$ and $Q$, we have the following lemma.
\begin{lemma}\label{converge}
For any $\d>0$, up to a subsequence, $Q^\e\ri Q$ in $H^{2-\d}_{loc}(\Om)$.
\begin{proof}
The argument is similar to the proof of Theorem \ref{thm: small A results}.

Since $f_b^{\e}$ are convex,
Proposition \ref{prop: higher interior regularity} applies to $Q^\e$ with uniform-in-$\e$ interior $H^2$-bound.
This implies that there exists $Q_*$ such that up to a subsequence, $Q^\e\ri Q_*$ strongly in $H^{2-\d}_{loc}(\Omega)$.
We may assume that $Q^\e\ri Q_*$, and $\nabla Q^\e\ri \nabla Q_*$ almost everywhere in $\Om$.
Hence,
\beqo
f_b^\e(Q^\e)\ri f_b(Q_*),\,f_e(\nabla Q^\e)\ri f_e(\nabla Q_*),\text{ a.e.\;in }\Om.
\eeqo
In justifying the first convergence, we need that $f_b^\e\ri f_b$ locally uniformly in $\BQ^{\mathrm{o}}_{phy}$.
By the minimality of $Q^\e$ and the assumption $f_b^\e(Q) \leq f_b(Q)$,
\beqo %\label{eqn: comparison of E eps and E}
E^\e[Q^\e] = \int_\Om f_e(\nabla Q^\e)+f_b^\e(Q^\e)\,dx\leq \int_\Om f_e(\nabla Q)+f_b^\e(Q)\,dx\leq E[Q].
\eeqo
By Fatou's Lemma,
\beqo
\int_\Om f_e(\nabla Q_*)+f_b(Q_*)\,dx\leq E[Q].
\eeqo
Hence, $Q_*=Q$ by the uniqueness of minimizer, which completes the proof.
\end{proof}

\end{lemma}

Now we are ready to derive an estimate for $f_b(Q)$.

\begin{lemma}\label{lemma: estimate for f_b}
Assume $A>-\f35$ and $p\in [1,p(A))$.
Suppose $f_b$ satisfies assumptions \eqref{assumption for f_b: smooth and convex}-\eqref{assumption for f_b: infinity}.

\begin{enumerate}
  \item If $f_b$ additionally satisfies \eqref{eqn: stronger assumptions inverse power case}, then for any $V\subset\subset \Om$,
\beq
\int_{V} |\na Q|^{2p-2}\f{|\na f_b(Q)|^2}{f_b(Q)} \,dx\leq C(V,A,p,E[Q],c_s).
\label{eqn: bound on gradient f_b}
\eeq

\item If $f_b$ additionally satisfies \eqref{eqn: stronger assumptions by EKT log case}, then for any $V\subset\subset \Om$,
\beq
\int_{V} |\na Q|^{2p-2}|\na f_b|^2 \,dx\leq C(V,A,p,E[Q],c_0).
\label{eqn: bound on gradient f_b log case}
\eeq
\end{enumerate}
\begin{proof}

Up to minor adaptations, arguments in this proof are in the same spirit as those in the proof of Proposition \ref{prop: higher interior regularity}.
See also \cite[Theorem 4.3]{Evans2016partial}.

\setcounter{step}{0}
\begin{step}\label{step: derive variation inequality}
Let $V\subset\subset W\subset\subset \Om$ and $\xi$ be defined as in the proof of Proposition \ref{prop: higher interior regularity}.
%Let $\{f_b^\e\}$ and $Q^\e$ be defined as in the beginning of this section.
Similar to the proof of Proposition \ref{prop: higher interior regularity}, define
$$
\tilde{u}^\e_k = D_k^{-h}(\xi^2 |D^h Q^\e|^{2p-2}D_k^h Q^\e),
$$
where
$$
D^h Q^\e : = (D^h_1 Q^\e,D^h_2 Q^\e,D^h_3 Q^\e)^\intercal,\quad |D^h Q^\e|^2 = \sum_{m = 1}^3 |D^h_m Q^\e|^2.
$$
Since $Q^\e\in H^2_{loc}(\Om)$ , $\tilde{u}_k^\e \in H^1_0(\Om, \BQ)$ for all $p\geq 1$.
By the minimality of $Q^\e$, we argue as before to obtain that
$$
\int_\Om \f{\pa f_e(\na Q^\e)}{\pa Q_{ij,l}}\pa_l \tilde{u}^\e_{k,ij} + D f_b^\e(Q^\e)\cdot \tilde{ u}^\e_k\,dx = 0,
$$
which gives
\beq
\int_\Om D_k^h\left(\f{\pa f_e(\na Q^\e)}{\pa Q_{ij,l}}\right)\pa_l (\xi^2 |D^h Q^\e|^{2p-2}D_k^h Q_{ij}^\e) + D_k^h (D f_b^\e(Q^\e))(\xi^2 |D^h Q^\e|^{2p-2}D_k^h Q^\e)  \,dx= 0.
\label{eqn: variational inequality in epsilon problem}
\eeq
By Lemma \ref{converge}, it is easy to justify that up to a subsequence, as $\e\to 0$,
\beq
\begin{split}
&\int_\Om D_k^h\left(\f{\pa f_e(\na Q^\e)}{\pa Q_{ij,l}}\right)\pa_l (\xi^2 |D^h Q^\e|^{2p-2}D_k^h Q_{ij}^\e)\,dx \\
&\quad \to
\int_\Om D_k^h\left(\f{\pa f_e(\na Q)}{\pa Q_{ij,l}}\right)\pa_l (\xi^2 |D^h Q|^{2p-2}D_k^h Q_{ij})\,dx.
\end{split}
\label{eqn: convergence of elastic energy}
\eeq
On the other hand, since $f_b^\e$ are convex, for all $w_1,w_2\in \BQ^{\mathrm{o}}_{phy}$,
$$
(Df_b^\e(w_1)-Df_b^\e(w_2))\cdot (w_1-w_2)\geq 0,
$$
which implies that
$$
\xi^2 D_k^h(Df_b^\e(Q^\e))D_k^hQ^\e \geq 0.
$$
It is assumed that $Df_b^\e\to Df_b$ locally uniformly in $\BQ^{\mathrm{o}}_{phy}$, and by Lemma \ref{converge}, up to a subsequence $Q^\e \to Q$ pointwise.
Hence,
$$
D_k^h(Df_b^\e(Q^\e))\to D_k^h(Df_b(Q))\quad \mbox{a.e.\;in }\Om.
$$
By Fatou's Lemma, for that subsequence,
$$
\int_\Om D_k^h (D f_b(Q))(\xi^2 |D^h Q|^{2p-2}D_k^h Q)  \,dx \leq \lim_{\e\to 0^+} \int_\Om D_k^h (D f_b^\e(Q^\e))(\xi^2 |D^h Q^\e|^{2p-2}D_k^h Q^\e)  \,dx.
$$
This together with \eqref{eqn: variational inequality in epsilon problem} and \eqref{eqn: convergence of elastic energy} implies that
$$
\int_\Om D_k^h\left(\f{\pa f_e(\na Q)}{\pa Q_{ij,l}}\right)\pa_l (\xi^2 |D^h Q|^{2p-2}D_k^h Q_{ij}) + D_k^h (D f_b(Q))(\xi^2 |D^h Q|^{2p-2}D_k^h Q)  \,dx\leq 0,
$$
\end{step}

\begin{step}\label{step: derive interior estimate}
Using the form of $f_e$ in \eqref{elastic}, we rewrite the inequality above as % if $p\in[1,p(A))$,
\beqo
\begin{split}
&\;\int_\Om  \xi^2 |D^h Q|^{2p-2}\cdot D_k^h(Q_{ij,l}+\d_{jl}A(\di Q)_i)\cdot D_k^hQ_{ij,l}\,dx\\
&\; +\int_\Om \xi^2 |D^h Q|^{2p-2}\cdot D_k^h (D f_b(Q))D_k^h Q  \,dx\\
\leq &\;-\int_\Om D_k^h(Q_{ij,l}+\d_{jl}A(\di Q)_i)D_k^h Q_{ij}\cdot 2\xi \pa_l \xi |D^h Q|^{2p-2}\,dx\\
&\;-(2p-2)\int_\Om D_k^h(Q_{ij,l}+\d_{jl}A(\di Q)_i)D_k^h Q_{ij}\cdot \xi^2 |D^h Q|^{2p-4}(\pa_lD^h Q:D^h Q)\,dx.
\end{split}
\eeqo
Here the summation convention applies to all repeated indices, including $k$.
With $p\geq 1$ and $A\geq -\f35$, we derive that
\beq
\begin{split}
&\;\int_\Om  \xi^2 |D^h Q|^{2p-2} |D^h\na Q|^2\,dx +\int_\Om \xi^2 |D^h Q|^{2p-2}\cdot D_k^h (D f_b(Q))D_k^h Q  \,dx\\
\leq &\;C\int_\Om |\xi||\na \xi| |D^h Q|^{2p-1}|D^h\na Q|\,dx+C\int_\Om \xi^2|D_k^h\na  Q||D_k^h Q| |D^h Q|^{2p-3}|D^h \na Q|\,dx\\
\end{split}
\label{eqn: estimate with bulk energy full finite difference}
\eeq
If $p\geq 2$,
$$
|D_k^h\na  Q||D_k^h Q||D^h Q|^{2p-3}|D^h \na Q| \leq C(|D_k^h\na  Q||D_k^h Q|^{p-1})^{\f{1}{p-1}}(|D^h \na Q||D^h Q|^{p-1})^{\f{2p-3}{p-1}}.
$$
Otherwise,
$$
|D_k^h\na  Q||D_k^h Q||D^h Q|^{2p-3}|D^h \na Q| \leq C|D_k^h\na  Q||D_k^h Q|^{p-1}\cdot |D^h\na  Q||D^h Q|^{p-1}.
$$
In either case, applying Young's inequality to the right hand side of \eqref{eqn: estimate with bulk energy full finite difference}, we obtain that

\beqo
\begin{split}
&\;\int_\Om  \xi^2 |D^h Q|^{2p-2} |D^h\na Q|^2\,dx +\int_\Om \xi^2 |D^h Q|^{2p-2}\cdot D_k^h (D f_b(Q))D_k^h Q  \,dx\\
\leq &\;C\int_\Om |\na \xi|^2 |D^h Q|^{2p}\,dx+C\int_\Om \xi^2|D_k^h\na  Q|^2 |D_k^h Q|^{2p-2}\,dx.
\end{split}
\eeqo
By \eqref{imp3}, for $p\in [1,p(A))$,
\beq
\int_\Om  \xi^2 |D^h Q|^{2p-2} |D^h\na Q|^2\,dx +\int_\Om \xi^2 |D^h Q|^{2p-2}\cdot D_k^h (D f_b(Q))D_k^h Q  \,dx\leq C,
\label{eqn: estimate with bulk energy crude form}
\eeq
where $C = C(V,A,p,E[Q])$.
\end{step}

\begin{step}
\label{step: derive bound for gradient of f_b}

Take an arbitrary $\d\ll 1$
and let $\Om_\d$ be defined in \eqref{eqn: def of lower level set of distance}.
Since $Q$ is H\"{o}lder continuous on $\overline{W}$ by Lemma \ref{converge}, $Q$ is separated away from $\pa \BQ_{phy}$ on an $h$-neighborhood of $V\backslash\Om_\d$ provided that $h\ll 1$.
Thus $Q$ is smooth in this neighborhood, which means on $V\backslash\Om_\d$,
$$
D^h Q\to \na Q,\quad D_k^h (D f_b(Q))\to \pa_k (D f_b(Q)).
$$
By \eqref{eqn: estimate with bulk energy crude form} and dominated convergence theorem, for any $\d\ll 1$,
\beq
\int_{V\backslash\Om_\d} |\na Q|^{2p-2}\pa_k (D f_b(Q))\pa_k Q  \,dx\leq C.
\label{eqn: bound on gradient f_b crude form}
\eeq
To this end, if we assume \eqref{eqn: stronger assumptions inverse power case}, on $V\backslash\Om_\d$,
$$
\pa_k(Df_b(Q))\pa_k Q = \f{\pa^2 f_b}{\pa Q_{ij}\pa Q_{mn}} \pa_k Q_{ij}\pa_k Q_{mn}\geq c_s \f{|Df_b\cdot \na Q|^2}{f_b(Q)} = \f{c_s |\na f_b(Q)|^2}{f_b(Q)}.
$$
Combining this with \eqref{eqn: bound on gradient f_b crude form} and sending $\d\to 0$, we obtain \eqref{eqn: bound on gradient f_b}.
\eqref{eqn: bound on gradient f_b log case} can be derived similarly if \eqref{eqn: stronger assumptions by EKT log case} is assumed.
\end{step}
\end{proof}
\end{lemma}

Recall that in the proof of Theorem \ref{thm1}, we used the distance function $d(Q)$ in the construction of the comparison configuration $Q_a$.
In the proof of Theorem \ref{thm: improved dimension bound}, we would like the comparison configuration to depend more explicitly on $f_b$ so that the estimate in Lemma \ref{lemma: estimate for f_b} may be used.
For this purpose, we need the following technical construction.

Let us first consider the case where $f_b$ satisfies the assumptions \eqref{assumption for f_b: growth} with \eqref{bulk} and \eqref{assumption for f_b: gradient bound} with \eqref{eqn: gradient bound inverse power case} and \eqref{eqn: stronger assumptions inverse power case}. %$\Om_a$ be defined as in \eqref{eqn: def of lower level set of distance}.
With $m_s$ given in \eqref{bulk}, define $\tilde{\eta}_a:[0,\infty)\ri\mathbb{R}_+$ such that
\begin{enumerate}
\item $\tilde{\eta}_a(y)$ is a decreasing $C^1$-function on $[0,+\infty)$.
\item
\beq
\tilde{\eta}_a(y)\begin{cases}
=1,& \mbox{ if }y\leq \f{m_s}{2a^s},\\
\leq \min\left\{1,\frac{1-\sqrt{6}a}{1-\sqrt{6}
\left(\f{m_s}{y}\right)^{\f{1}{s}}}\right\},&\mbox{ if }y\in \left[ \f{m_s}{2a^s},\f{2m_s}{a^s}\right],\\
=\frac{1-\sqrt{6}a}{1-\sqrt{6}
\left(\f{m_s}{y}\right)^{\f{1}{s}}},&\mbox{ if }y\geq \f{2m_s}{a^s}.
\end{cases}
\label{eqn: def of eta_a}
\eeq
\item For all $y\geq \f{m_s}{2a^s}$,
\beq
|\tilde{\eta}'_a(y)|\leq \f{Cm_s}{s}y^{-\f{1}{s}-1}.
\label{eqn: derivative of eta_a}
\eeq
\end{enumerate}
It is not difficult to construct such $\tilde{\eta}_a$.
Roughly speaking, $\tilde{\eta}_a(y)$ can be viewed as a smoothed version of
$$
\min\left\{1, \frac{1-\sqrt{6}a}{1-\sqrt{6}\left(\f{m_s}{y}\right)^{\f{1}{s}}}\right\}.
$$
When $a\ll 1$, it is easy to verify that $|1-\tilde{\eta}_a| \leq Ca$ and $\tilde{\eta}_a\circ f_b$ is $C^1$ on $\BQ^{\mathrm{o}}_{phy}$ with bounded gradient.
Indeed,
combining \eqref{eqn: def of eta_a}, \eqref{eqn: derivative of eta_a} with \eqref{bulk} and \eqref{eqn: gradient bound inverse power case}, we deduce that % and \eqref{eqn: derivative of eta_a},
for all $w\in \BQ^{\mathrm{o}}_{phy}$,
\beq
|\tilde{\eta}_a'(f_b(w))\cdot Df_b(w)|\leq \f{Cm_s}{s}\cdot C_s^{\f{1}{s}}.
\label{eqn: deriveative of eta_a composed with f_b}
\eeq
Define $\tilde{h}_a(Q):= \tilde{\eta}_a(f_b(Q))Q$.
It can be verified that when $a\ll 1$,
\begin{itemize}
\item
$\tilde{h}_a(Q) \equiv Q$ if $d(Q)\geq \Lam_s a$.

This directly follows from \eqref{bulk} and \eqref{eqn: def of eta_a}.
\item
$\tilde{h}_a$ retracts $\BQ_{phy}$ to a smaller subset of $\BQ^\mathrm{o}_{phy}$, such that $d(\tilde{h}_a(Q))\geq a$.
Indeed, by Lemma \ref{lemma: characterization of d(Q)} (in particular, \eqref{formula for d(Q)}),
$$
d(\tilde{h}_a(Q)) = \f{\sqrt{6}}{2}\left(\tilde{\eta}_a(f_b(Q))\lambda_1(Q)+\f13\right) = \f{\sqrt{6}}{6}-\tilde{\eta}_a(f_b(Q))\left(\f{\sqrt{6}}{6}-d(Q)\right).
$$
In order to show $d(\tilde{h}_a(Q))\geq a$, it suffices to notice that
$$
\tilde{\eta}_a(f_b(Q))\leq \f{1-\sqrt{6}a}{1-\sqrt{6}d(Q)},
$$
which is true because of \eqref{bulk} and \eqref{eqn: def of eta_a}.

\item $\tilde{h}_a(Q)\in H^1(\Om)$.
This follows from $C^1$-regularity of $\tilde{\eta}_a\circ f_b$ and \eqref{eqn: deriveative of eta_a composed with f_b}.

\end{itemize}

To this end, we are ready to show part \eqref{part: inverse power case} of Theorem \ref{thm: improved dimension bound}.

\begin{proof}[Proof of part \eqref{part: inverse power case} of Theorem \ref{thm: improved dimension bound}]
We proceed as in the proof of Theorem \ref{thm1} with minor modifications.
\setcounter{step}{0}
\begin{step}
Let $U\subset\subset V\subset\subset \Omega$ and a cut-off function $\r$ be defined as in the proof of Theorem \ref{thm1}.

Let
$$
\tilde{Q}_a(x) = \rho \tilde{h}_a(Q) + (1-\rho)Q.
$$
By the argument above, $\tilde{Q}_a\in \mathcal{A}$ when $a\ll 1$.

We shall use $\tilde{Q}_a$ as a comparison map in \eqref{energy}.
By the minimality of $Q$, % in \eqref{energy},
\beq\label{eqn: comparing tilde Q_a and Q}
\int_{\Om}f_b(Q(x))-f_b(\tilde{Q}_a(x))\,dx\leq \int_{\Om}f_e(\nabla \tilde{Q}_a(x))-f_e(\nabla Q(x))\,dx.
\eeq
With $a\ll 1$, we may assume that $f_b(0)\leq f_b(Q)$ for any $Q\in \BQ^{\mathrm{o}}_{phy}$ with $d(Q) \leq\Lam_s a$.
For the same reason as in the proof of Theorem \ref{thm1}, $f_b(Q)\geq f_b(\tilde{Q}_a)$ for all $x\in \Omega$.
Hence, we can still show that
\beq
\int_{\Om}f_b(Q(x))-f_b(\tilde{Q}_a(x))\,dx
\geq  \frac{m_s}{2}\cdot  \left(\frac{a}{\Lam_s}\right)^{-s}|\Omega_{a/\Lam_s}\cap U|,
\label{eqn: lower bound of LHS with tilde Q a}
\eeq
where $\Lam_s$ is defined in \eqref{eqn: def of Lambda s}.

\end{step}
\begin{step}
On the right hand side of \eqref{eqn: comparing tilde Q_a and Q}, % in a better way.
since $\tilde{Q}_a\equiv Q$ outside $\Om_{\Lam_s a}\cap V$,
\beq
\begin{split}
&\;\left|\int_{\Om}f_e(\nabla \tilde{Q}_a(x))-f_e(\nabla Q(x))\,dx\right|\\
\leq &\; C\int_{\Omega_{\Lam_s a}\cap V}|\nabla (\tilde{Q}_a-Q)|^2 + |\nabla Q||\nabla (\tilde{Q}_a-Q)|\,dx\\
\leq &\; C\int_{\Omega_{\Lam_s a}\cap V}a|\na Q|^2 +a|\na \r|^2+|\na Q||\na (\tilde{\eta}_a(f_b(Q)))|\,dx.
\end{split}
\label{eqn: bound elastic energy change}
\eeq
Here we used the definition of $\tilde{h}_a(Q)$ to derive that
$$
|\nabla (\tilde{Q}_a-Q)|\leq C(a|\na Q|+a|\na \r|+|\na (\tilde{\eta}_a(f_b(Q)))|).
$$

Now we proceed in two different cases.

\setcounter{case}{0}
\begin{case}\label{case: p can not be big}
Suppose $p(A)\leq 2$, i.e., we may only take $p< 2$ in Lemma \ref{lemma: estimate for f_b}.
Thanks to \eqref{eqn: derivative of eta_a}, \eqref{eqn: deriveative of eta_a composed with f_b} and \eqref{eqn: bound elastic energy change}, with $q = 6p$ and $\d\ll 1$,
\beqo
\begin{split}
&\;\left|\int_{\Om}f_e(\nabla \tilde{Q}_a(x))-f_e(\nabla Q(x))\,dx\right|\\
\leq &\; Ca|\Om_{\Lam_s a}\cap V|^{1-\f{2}{q}}+ C\int_{\Omega_\d\cap V}|\na Q|^2\,dx+ C\int_{(\Omega_{\Lam_s a}\backslash \Om_\d)\cap V}|\na Q||\tilde{\eta}_a'(f_b)||\na f_b|\,dx\\
\leq &\; Ca|\Om_{\Lam_s a}\cap V|^{1-\f{2}{q}}+ C|\Om_\d \cap V|^{1-\f2q}+ C\int_{(\Omega_{\Lam_s a}\backslash \Om_\d)\cap V}|\na Q| \cdot f_b^{-(1+\f1s)}|\na f_b|\,dx.
\end{split}
\eeqo
Combining this with \eqref{eqn: comparing tilde Q_a and Q}, \eqref{eqn: lower bound of LHS with tilde Q a} and the assumption \eqref{assumption for f_b: growth}, and sending $\d\to 0$, %we find that
\beqo
|\Omega_{a/\Lam_s}\cap U|
\leq  Ca^{1+s}|\Om_{\Lam_s a}\cap V|^{1-\f{1}{3p}}
+ Ca^{1+\f{3s}{2}}\int_{\Omega_{\Lam_s a}\cap V}|\na Q| \cdot f_b^{-\f12}|\na f_b|\,dx.
\eeqo
By Proposition \ref{prop: higher interior regularity}, Lemma \ref{lemma: estimate for f_b} and
H\"{o}lder's inequality,
\beq
\begin{split}
&\;\int_{\Omega_{\Lam_s a}\cap V}|\na Q| \cdot f_b^{-\f12}|\na f_b|\,dx\\
\leq &\;
C\|\na Q\|_{L^{6p}(V)}^{2-p}\left\||\na Q|^{p-1} \cdot f_b^{-\f12}|\na f_b|\right\|_{L^2(V)}
|\Omega_{\Lam_s a}\cap V|^{\f{1}{2}-\f{2-p}{6p}}\\
\leq &\; C|\Omega_{\Lam_s a}\cap V|^{\f{1}{2}-\f{2-p}{6p}}.
\end{split}
\eeq
Hence,
\beqo
|\Omega_{a/\Lam_s}\cap U|
\leq  Ca^{1+s}|\Om_{\Lam_s a}\cap V|^{1-\f{1}{3p}}
+ Ca^{1+\f{3s}{2}}|\Omega_{\Lam_s a}\cap V|^{\f{1}{2}-\f{2-p}{6p}}.
\eeqo
By a boot-strapping argument similar to the one in the proof of Theorem \ref{thm1}, we can show that for any
\beqo
\beta < \f{3p(A)(3s+2)}{2(p(A)+1)},
%\label{eqn: improved decay exponent of size of upper level set small p}
\eeqo
and any $V\subset \subset \Om$, we have $|\Omega_a\cap V|\leq Ca^\beta$ for all $a\ll 1$.
Hence, \eqref{eqn: improved dimension bound inverse power case} can be proved for the case $p(A)\leq 2$ by arguing as in the proof of Theorem \ref{thm1}.
\end{case}

\begin{case}
Now suppose $p(A)>2$, i.e., we may take $p \geq 2$ in Lemma \ref{lemma: estimate for f_b}.
Again by the assumption \eqref{assumption for f_b: growth} on $f_b$, \eqref{eqn: derivative of eta_a}, \eqref{eqn: deriveative of eta_a composed with f_b} and \eqref{eqn: bound elastic energy change},
\beqo
\begin{split}
&\;\left|\int_{\Om}f_e(\nabla \tilde{Q}_a(x))-f_e(\nabla Q(x))\,dx\right|\\
\leq &\; Ca|\Om_{\Lam_s a}\cap V|^{1-\f{2}{q}}+ C\int_{\Omega_\d\cap V}|\na Q|^2\,dx+ C\int_{(\Omega_{\Lam_s a}\backslash \Om_\d)\cap V}|\na Q|^{2-\f2p}(|\tilde{\eta}_a'(f_b)||\na f_b|)^{\f2p}\,dx\\
\leq &\; Ca|\Om_{\Lam_s a}\cap V|^{1-\f{2}{q}}+ C|\Om_\d \cap V|^{1-\f2q}+ Ca^{\f{2+s}{p}}\int_{(\Omega_{\Lam_s a}\backslash \Om_\d)\cap V}|\na Q|^{2-\f2p} \left(f_b^{-\f12}|\na f_b|\right)^{\f2p}\,dx.
\end{split}
\eeqo
Combining this with \eqref{eqn: comparing tilde Q_a and Q} and \eqref{eqn: lower bound of LHS with tilde Q a}, sending $\d\to 0$, and applying H\"{o}lder's inequality as before, we find that
\beqo
|\Om_{a/\Lam_s}\cap U|
\leq Ca^{1+s}|\Om_{\Lam_s a}\cap V|^{1-\f{1}{3p}}+ Ca^{s+\f{2+s}{p}}|\Om_{\Lam_s a}\cap V|^{1-\f1p}.
\eeqo
Arguing as in Case \ref{case: p can not be big}, for any
$$
\beta < 2+s+sp(A),
$$
and any $V\subset \subset \Om$, we have $|\Omega_a\cap V|\leq Ca^\beta$ for all $a\ll 1$.
Then \eqref{eqn: improved dimension bound inverse power case} can be proved as before for the case $p(A)>2$.
\end{case}
\end{step}
\end{proof}

Next, we consider the case when $f_b$ satisfies the assumptions \eqref{assumption for f_b: growth} with \eqref{eqn: log growth non-uniform} and \eqref{assumption for f_b: gradient bound} with \eqref{eqn: stronger assumptions by EKT log case} and \eqref{eqn: gradient bound log case}.
The argument is almost parallel.

%Let $z(y)$ be a smooth function defined on $[k_0|\ln a|+m_0,+\infty)$, satisfying the following conditions.
%\begin{enumerate}
%  \item For all $y\geq k_0|\ln a|+m_0$,
%\beq
%z(y)\geq \sup_{Q\in f_b^{-1}(y)}d(Q).
%\eeq
%If the set $f_b^{-1}(y)$ is empty, the supremum is understood as $-\infty$.
%
%\item For some universal constant $C>0$, and arbitrary $y\geq k_0|\ln a|+m_0$,
%\beq
%|z'(y)|\cdot \sup_{Q\in f_b^{-1}(y)}|Df_b|\leq C.
%\eeq
%
%\end{enumerate}
%
%\textcolor{red}{One possible choice is $z(y)\equiv \exp(K_0^{-1}(k_0 \ln a+M_0-m_0))$.}

Take $a\ll 1$.
Let $\Lambda_0 := \exp(1+k_0^{-1}(M_0-m_0))$.
Take $\hat{\eta}_a$ as a decreasing $C^1$-function on $[0,+\infty)$,
\beqo
\hat{\eta}_a(y)=
\begin{cases}
  1, & \mbox{if } y\leq  k_0|\ln a|+m_0- \exp(m_0/k_0),\\
  \frac{1-\sqrt{6}\Lambda_0 a}{1-\sqrt{6}a}, & \mbox{if } y\geq k_0|\ln a|+m_0.
\end{cases}
\eeqo
and for $y\in [k_0|\ln a|+m_0 -\exp(m_0/k_0), k_0|\ln a|+m_0]$, $\hat{\eta}_a$ needs to satisfy
\beq
|\hat{\eta}_a'(y)|\leq 4\Lam_0 \exp(-k_0^{-1}y).
\label{eqn: derivative bound for eta}
\eeq
This is achievable since for $a\ll 1$,
\beqo
\exp(m_0/k_0)\cdot 4\Lam_0 \exp(-k_0^{-1}(k_0|\ln a|+m_0))> 1-\f{1-\sqrt{6}\Lam_0 a}{1-\sqrt{6}a}.
\eeqo
We also note that \eqref{eqn: derivative bound for eta} implies
\beq
\hat{\eta}_a'\leq Ca
\label{eqn: derivative bound for eta simplified}
\eeq
for some universal constant $C$.

We claim that
\begin{enumerate}
  \item
  $\hat{\eta}_a\circ f_b$ is Lipschitz continuous in $\mathcal{Q}_{phy}^{\mathrm{o}}$.
  Indeed, by \eqref{eqn: gradient bound log case} and \eqref{eqn: derivative bound for eta},
  \beqo
  |D(\hat{\eta}_a \circ f_b)| \leq |\hat{\eta}_a'(f_b)| |Df_b|\leq 4\Lam_0 C_0.
  \eeqo

  \item $|1-\hat{\eta}_a|\leq C a$ for some universal $C>0$, thanks to the monotonicity of $\hat{\eta}_a$ and $a\ll 1$.
  \item Whenever $d(Q)\leq a$,
  \beq
  f_b(Q) - f_b(\hat{\eta}_a\circ f_b(Q)Q)\geq k_0.
  \label{eqn: f_b get decreased}
  \eeq
  Indeed, by \eqref{eqn: log growth non-uniform} and \eqref{formula for d(Q)}, % and \eqref{eqn: bound for k(Q)},
  \beqo
  \begin{split}
  &\;f_b(Q) -f_b(\hat{\eta}_a\circ f_b(Q)Q) \\
  \geq &\;k(Q)|\ln d(Q)|+m_0 - k(Q)|\ln d(\hat{\eta}_a\circ f_b(Q)Q)|-M_0\\
  %\geq &\;k(Q)\ln \frac{d(\hat{\eta}_a\circ f_b(Q)Q)}{d(Q)}-(M_0-m_0)\\
%  =&\; k(Q)\ln \f{\f{\sqrt{6}}{2}\hat{\eta}_a\circ f_b(Q)\lambda_1 (Q)+\f{\sqrt{6}}{6}}{d(Q)}-(M_0-m_0)\\
  \geq &\; k_0\ln \f{\hat{\eta}_a\circ f_b(Q)(d(Q)-\f{\sqrt{6}}{6})+\f{\sqrt{6}}{6}}{d(Q)}-(M_0-m_0).
  \end{split}
  \eeqo
  Since $d(Q)\leq a$ implies $f_b(Q)\geq k_0 |\ln a|+m_0$, we know that $\hat{\eta}_a  = \f{1-\sqrt{6}\Lam_0 a}{1-\sqrt{6}a}$.
  Hence,
 \beqo
  \begin{split}
  &\;f_b(Q) -f_b(\hat{\eta}_a\circ f_b(Q)Q) \\
  \geq &\; k_0\ln \f{\f{1-\sqrt{6}\Lam_0 a}{1-\sqrt{6}a}(a-\f{\sqrt{6}}{6})+\f{\sqrt{6}}{6}}{a}-(M_0-m_0)\\
  = &\; k_0\ln \Lam_0-(M_0-m_0)\\
  =&\;k_0.
  \end{split}
  \eeqo
  In particular, for arbitrary $\nu >1$, \eqref{eqn: f_b get decreased} holds when $f_b (Q)\geq \nu K_0|\ln a|\geq K_0|\ln a|+M_0$ as long as $a\ll 1$ (with the needed smallness depending on $\nu$, $K_0$ and $M_0$).
\end{enumerate}

Define $\hat{h}_a(Q):= \hat{\eta}_a(f_b(Q))Q$.
It is not difficult to justify that for $a\ll 1$,
\begin{itemize}
\item
$\hat{h}_a(Q) \equiv Q$ if $f_b(Q)\leq k_0|\ln a|+m_0-\exp(m_0/k_0)$. %$d(Q)\geq (\Lam_0 a)^{k_0/K_0}$.
\item
%$\hat{h}_a$ retracts $\BQ_{phy}$ to a smaller subset of $\BQ^{\mathrm{o}}_{phy}$, such that $d(\hat{h}_a(Q))\geq a$.
$f_b(\hat{h}_a(Q))\leq f_b(Q)$. This follows from convexity of $f_b$.

\item $\hat{h}_a(Q)\in H^1(\Om)$.
This follows from $C^1$-regularity of $\hat{\eta}_a\circ f_b$ and boundedness of $\hat{\eta}_a'(f_b) \cdot Df_b$.
\end{itemize}

\begin{proof}[Proof of part \eqref{part: log case} of Theorem \ref{thm: improved dimension bound}]
Define $\hat{Q}_a$ as before, $\hat{Q}_a (x) = \r\hat{h}_a(Q)+(1-\r)Q$.
We still have
\beqo%\label{eqn: comparing tilde Q_a and Q}
\int_{\Om}f_b(Q(x))-f_b(\hat{Q}_a(x))\,dx\leq \int_{\Om}f_e(\nabla \hat{Q}_a(x))-f_e(\nabla Q(x))\,dx.
\eeqo
Denote
\beqo
S_y = \{x\in \Om:\;f_b(Q(x))\geq y\}.
\eeqo

For $a\ll 1$, by \eqref{eqn: f_b get decreased} and the remark following that,
\beqo
\int_{\Om}f_b(Q(x))-f_b(\hat{Q}_a(x))\,dx
\geq  k_0|S_{\nu K_0 |\ln a|}\cap U|.
%\label{eqn: lower bound of LHS with tilde Q a log case}
\eeqo
On the other hand,
\beqo
\begin{split}
&\;\left|\int_{\Om}f_e(\nabla \hat{Q}_a(x))-f_e(\nabla Q(x))\,dx\right|\\
%\leq &\; C\int_{\Omega_{\Lam_s a}\cap V}|\nabla (\hat{Q}_a-Q)|^2 + |\nabla Q||\nabla (\hat{Q}_a-Q)|\,dx\\
\leq &\; C\int_{S_{k_0|\ln a|+m_0-\exp(m_0/k_0)}\cap V}a|\na Q|^2 +a|\na \r|^2+|\na Q||\na (\hat{\eta}_a(f_b(Q)))|\,dx.
\end{split}
%\label{eqn: bound elastic energy change log case}
\eeqo
Take arbitrary $\theta < k_0/(\nu K_0)< 1$.
For $a\ll 1$ with the smallness depending on $\theta$, we have
$$
\theta (\nu K_0 |\ln a|)\leq k_0|\ln a|+m_0-\exp(m_0/k_0).
$$

Suppose $p(A)\leq 2$.
Arguing as before, we derive that for $a\ll 1$,
\beqo
\begin{split}
|S_{\nu K_0 |\ln a|}\cap U|
\leq &\; Ca|S_{\theta \nu K_0 |\ln a|}\cap V|^{1-\f{2}{q}}
+ Ca\int_{S_{\theta \nu K_0 |\ln a|}\cap V}|\na Q|^{2-p} \cdot |\na Q|^{p-1}|\na f_b|\,dx\\
\leq &\; Ca|S_{\theta \nu K_0 |\ln a|}\cap V|^{1-\f{1}{3p}}
+ Ca|S_{\theta \nu K_0 |\ln a|}\cap V|^{\f{1}{2}-\f{2-p}{6p}}.
\end{split}
\eeqo
Here we used \eqref{eqn: derivative bound for eta simplified}.
Hence,
\beqo
|S_{\nu K_0|\ln a|}\cap U|
\leq Ca|S_{\nu K_0|\ln a^\theta|}\cap V|^{\f{2}{3}-\f{1}{3p}}.
\eeqo
By boot-strapping, for any
\beqo
\beta < \sum_{j = 0}^\infty \left[\theta\left(\f23-\f{1}{3p(A)}\right)\right]^j = \left[1-\theta\left(\f23-\f{1}{3p(A)}\right)\right]^{-1},
%\label{eqn: improved decay exponent of size of upper level set small p}
\eeqo
and any $V\subset \subset \Om$, we have $|S_{\nu K_0 |\ln a|}\cap V|\leq Ca^\beta$ for all $a\ll 1$.
By \eqref{eqn: log growth non-uniform}, this implies
\beqo
|\Om_{\exp(k_0^{-1}(\nu K_0 \ln a+m_0))}\cap V|\leq Ca^\b,
\eeqo
or equivalently,
$|\Om_{a}\cap V|\leq Ca^{k_0\beta/(\nu K_0)}$,
with $C$ being a constant independent of $a$.
This leads to \eqref{eqn: improved dimension bound log case} in the case $p(A)\leq 2$.

Finally, suppose $p(A)>2$. %, i.e., we may take $p \geq 2$ in Lemma \ref{lemma: estimate for f_b}.
Similarly,
\beqo
\begin{split}
|S_{\nu K_0 |\ln a|}\cap U|\leq &\; Ca|S_{\theta \nu K_0 |\ln a|}\cap V|^{1-\f{1}{3p}} + Ca^{\f{2}{p}}\int_{S_{\theta \nu K_0 |\ln a|}\cap V}|\na Q|^{2-\f2p} |\na f_b|^{\f2p}\,dx\\
\leq &\; Ca|S_{\theta \nu K_0 |\ln a|}\cap V|^{1-\f{1}{3p}}+ Ca^{\f{2}{p}}|S_{\theta \nu K_0 |\ln a|}\cap V|^{1-\f1p}\\
\leq &\; Ca^{\f{2}{p}}|S_{\nu K_0 |\ln a^\theta|}\cap V|^{1-\f1p}\\
\end{split}
\eeqo
Hence, for any
$\beta < 2/(p-\theta(p-1))$,
and any $V\subset \subset \Om$, we have $|\Omega_a\cap V|\leq Ca^{k_0\beta/(\nu K_0)}$ for all $a\ll 1$.
Then \eqref{eqn: improved dimension bound log case} for the case $p(A)>2$ follows.
\end{proof}

\section{Boundary Partial Regularity}
\label{section: boundary partial regularity}
This section is devoted to proving Theorem \ref{thm: boundary partial regularity}.
First we introduce some notations.
%To state that, we introduce some notations.
Let $x_0= (x_0^1, x_0^2, x_0^3)\in \mathbb{R}^3$ and $R>0$.
Define
\begin{align*}
B(x_0,R) =&\; \{x=(x^1,x^2,x^3)\in \BR^3: |x-x_0|<R\},\\
B^+(x_0,R) =&\; \{x\in B(x_0,R): x^3>x_0^3\},\\
\Gamma(x_0,R) = &\;B(x_0,R) \cap \{x^3 = x_0^3\},\\
Q_{x_0,R} =&\;\fint_{B^+(x_0,R)} Q(x)\,dx.
\end{align*}

In the special case $x_0=0$, we write them as $B_R$, $B_R^+$, $\Gamma_R$, and $Q_R$, respectively.

In order to prove $Q$ is H\"{o}lder continuous near some point $y_0\in \pa\Om$, we first use a smooth local diffeomorphism $\p^{-1}$ to a ball $U$ centered at $y_0$ to flatten the boundary, such that $y_0$, $U\cap \Om$, and $U\cap \pa \Om$ are mapped to $0$, $B_R^+$, and $\Gamma_R$, respectively.
Moreover, up to a rotation, we may assume that
\beq
\lim_{r\ri 0}\|\nabla (\p^{-1})-Id\|_{L^\infty(B_r^+)} =0.
\eeq
In other words, if we zoom in to smaller and smaller neighborhood of $y_0$, the deformation induced by $\p^{-1}$ is almost negligible, and $\p^{-1}$ behaves like an identity map.

To this end, under the change of variables $y = \p(x)$, define
\beq
\begin{split}
E[Q,U]:= &\;\int_U f_e(\nabla Q)+f_b(Q)\,dy\\
= &\;\int_{B_R^+}[f_e(\nabla (Q\circ \p)(x)(\nabla\p(x))^{-1})+f_b(Q\circ\p(x))]\det \nabla \p(x) \,dx.
\end{split}
\eeq
Hence, with abuse of notations, we turn to study minimizers of the following functional in a more general form
\beq\label{bdyfunc}
E[Q,B_R^+]=\int_{B_R^+} [f_e(J(x)\nabla Q)+f_b(Q)]g(x)\,dx,
\eeq
subject to Dirichlet data $Q = Q_0$ on $\pa B_R^+$ for some $Q_0\in C^\infty(\bar{B}_R, \BQ^\mathrm{o}_{phy})$.
Here $f_e$ is defined as in \eqref{elastic}, and $J(x):B_R^+\rightarrow \mathbb{R}^{3\times 3}$ is a smooth function, satisfying
\beq\label{eqn: assumption on J}
\lim_{r\rightarrow 0^+}\|J(x)-Id\|_{L^\infty(B_r^+)} = 0.
\eeq
$g(x):B_R^+\rightarrow \mathbb{R}$ is a smooth function, satisfying
\beq\label{eqn: assumption on g}
\lim_{r\rightarrow 0^+}\|g(x)-1\|_{L^\infty(B_r^+)} = 0.
\eeq
We also used the notation that $(J(x)\nabla Q)_{ij,k} \triangleq J_{kl}(x)\nabla_l Q_{ij}$.

For $x_0\in \G_R$ and $0<r<\mathrm{dist}(x_0,\pa\G_R)$, define an scaling-invariant quantity
\beqo
A_{x_0,r}=\f{1}{r}\int_{B^+(x_0,r)} |\na Q(x)|^2\,dx.
\eeqo
Denote $A_r:=A_{0,r}$.
Then we have
\begin{proposition}\label{pabdy}
Let $Q\in H^1(B_R^+, \BQ_{phy})$ be the unique minimizer of \eqref{bdyfunc} subject to smooth Dirichlet boundary condition $Q=Q_0$ on $\pa B_R^+$.
There exists $\e>0$, such that if for $x_0\in\G_R$, $\liminf_{r\ri 0}A_{x_0,r}< \e^2$, then $Q(x)$ is H\"{o}lder continuous in a neighborhood of $x_0$.
\end{proposition}

Theorem \ref{thm: boundary partial regularity} follows immediately from Proposition \ref{pabdy}.

\begin{proof}[Proof of Theorem \ref{thm: boundary partial regularity}]
The diffeomorphism $\psi^{-1}$ for flattening the boundary is smooth, and it is sufficiently close to an identity map (up to a rotation in general) if we only consider sufficiently small boundary patches.
Hence, a statement similar to Proposition \ref{pabdy} holds for the minimizer of \eqref{energy} in $\Omega$ with curved boundary.
This together with a classic covering argument implies that the minimizer of \eqref{energy} is H\"{o}lder continuous up to $\partial\Om\backslash S$, with $S\subset \partial \Om$ satisfying $\mathcal{H}^1(S) = 0$.

Since $Q_0(x)\in \BQ^\mathrm{o}_{phy}$ for all $x\in \partial\Om$, by the continuity, $\overline{\mathcal{C}}\cap \partial \Om \subset S$.
This completes the proof.
\end{proof}

The rest of this section is devoted to proving Proposition \ref{pabdy}.
The proof closely follows the classical variational proof in \cite{evans1987blowup} and the proof of interior partial regularity in \cite{Evans2016partial}, with necessary modifications to handle the boundary data.
In what follows, without loss of generality, we shall assume $x_0 = 0$.
We first prove a useful lemma.

\begin{lemma}\label{Lebesguepoint}
For all $r\in (0,R)$,
\beqo
|Q_{r}-Q_0(0)|\leq C\sqrt{A_r+r^2},
\eeqo
where $C$ is a constant only depending on $\|\nabla Q_0\|_{L^\infty (B_r^+)}$.

\begin{proof}
Since $Q-Q_0 =0$ on $\G_R$, by Poincar\'{e} inequality on domains with finite width,
\beqo
\begin{split}
A_r\geq &\;\f{1}{2r}\int_{B_r^+}|\na (Q-Q_0)|^2\,dx-\f{1}{r}\int_{B_r^+}|\na Q_0|^2\,dx\\
\geq &\; \f{1}{2r^3}\int_{B_{r}^+}|Q-Q_0|^2\,dx-\f{1}{r}\int_{B_r^+}|\na Q_0|^2\,dx\\
\geq &\; \f{\pi}{3} |(Q-Q_0)_{r}|^2-\f{1}{r}\int_{B_r^+}|\na Q_0|^2\,dx.
\end{split}
\eeqo
Hence, by the smoothness of $Q_0$,
\beq\label{ineq}
|(Q-Q_0)_{r}|\leq C\sqrt{A_r+r^2}.
\eeq
In addition,
$$
|(Q_0)_r - Q_0(0)|\leq Cr.
$$
Then the desired estimate follows.
\end{proof}
\end{lemma}

To this end, we shall prove the so-called small energy regularity. %, which implies Proposition \ref{pabdy} immediately.
\begin{lemma}\label{smallene}
Let $Q$ be the minimizer defined in Proposition \ref{pabdy}.
There exists $\theta\in(0,\f14)$ and $\e>0$, such that if
\beqo
r<\e,\quad A_{r}\leq \e^2,
\eeqo
then for any $x\in \Gamma_{r/4}$,
\beqo
A_{x,\theta r}\leq \f12 A_{r}.
\eeqo

\end{lemma}

We shall prove Lemma \ref{smallene} by contradiction.
Suppose the statement is false.
Then for a fixed $\theta\in (0,1/4)$ which we will determine later, there exists $\{(\e_i,r_i,x_i)\}_{i=1}^\infty$ such that
\beqo
\e_i\ri 0,\quad r_i\leq \e_i,\quad A_{r_i}\leq \e_i^2,\quad\mbox{ and }x_i\in \Gamma_{r_i/4},
\eeqo
while
$$
A_{x_i,\theta r_i}> \f12 A_{r_i}.
$$

Define
\beqo
Q_i(x)=\e_i^{-1}(Q(r_ix)-Q_{r_i}). %\quad \text{for } x\in B_1^+.
\eeqo
It is straightforward to verify that
$$
\int_{B_1^+}|\na Q_i|^2\,dx\leq 1,\quad (Q_i)_1=0,
$$
and $Q_i$ minimizes %the functional
\beq\label{eqn: rescaled local functional}
W_i[Q]=\int_{B_1^+}\left[f_e\left(\f{\e_i}{r_i}J(r_i x)\na Q\right)+f_b(\e_iQ+Q_{r_i})\right]g(r_i x)\,dx.
\eeq
Passing to a subsequence if necessary, there exists $x_*\in \overline{\Gamma_{1/4}}$, and $\tl{Q}\in H^1(B_1^+,\BQ)$, such that
\begin{itemize}
  \item $x_i/r_i\ri x_*$;
  \item $Q_i\rightharpoonup \tilde{Q}$ weakly in $H^1(B_1^+,\BQ)$, and strongly in $L^2(B_1^+,\BQ)$;
  \item $(\tl{Q})_1=0$;
  \item $Q_i|_{\G_{1}}\ri \tl{Q}|_{\G_1}$ in $C^2(\G_1)$. %\label{bdyconv}
\end{itemize}
Indeed, it suffices to verify the last claim.
Note that by \eqref{ineq},
\beqo
|(Q-Q_0)_{r_i}|\leq C\sqrt{A_{r_i}+r_i^2}\leq C\e_i.
\eeqo
By the definition of $Q_i$, for $\forall\,x\in \Gamma_1$,
\beqo
|Q_i(x)|\leq \e_i^{-1}|{Q_0}(r_i x)-(Q_0)_{r_i}|+\e_i^{-1}|(Q_0-Q)_{r_i}|\leq C(\e_{i}^{-1} r_i+1)\leq C.
\eeqo
%Hence, $Q_i|_{\G_{1}}$ is continuous and uniformly bounded. % for all $i\in \mathbb{N}^+$.
Moreover, for any $k\in \mathbb{N}$,% we compute the derivative on the boundary
\beqo
|\na_{\G}^k Q_i(x)|\leq \e_i^{-1}r_i^k\|\na_{\G}^kQ_0\|_{L^\infty(\Gamma_1)}\leq C_k.
\eeqo
The convergence in $C^2(\Gamma_1)$ follows from Arzel\`{a}-Ascoli lemma.

The next lemma shows that the $H^1$-convergence of $Q_i$ to $\tilde{Q}$ is in fact in the strong sense in smaller boundary patches.

\begin{lemma}\label{H1strong}
$\na Q_i$ converges to $\na \tl{Q}$ strongly in $L^2(B_r^+,\BQ)$ for any $0<r<1$.
\end{lemma}
\begin{proof}
Define Radon measures $\mu_i$ as follows,
\beqo
\mu_i(D):=\int_D |\na Q_i|^2+|\na \tl{Q}|^2\,dx \text{ for any measurable set }D\subset B_1^+.
\eeqo
Up to a subsequence, there exists a Radon measure $\mu$ such that $\mu_i\rightharpoonup \mu$ in the sense of measures. For all but countably many $r\in(0,1]$, it holds that $\mu(\pa B_r\cap B_1^+)=0$.
It suffices to show
$$
\lim_{i\ri \infty}\int_{B_r^+}|\na Q_i-\na \tl{Q}|^2\,dx=0
$$
for any $r\in(0,1)$ such that $\mu(\pa B_r\cap B_1^+)=0$.

Since $f_e$ is strictly convex and quadratic in $\nabla Q$, there exists $\lambda>0$ such that
\begin{align*}
&\;\int_{B_r^+}\left[f_e(\na Q_i(x))-f_e(\na \tl{Q}(x))\right]dx\\
\geq &\;\int_{B_r^+}\left[Df_e(\na \tl{Q}(x))\cdot (\na Q_i-\na \tl{Q})\right]dx+\lambda \int_{B_r^+}|\na Q_i-\na \tl{Q}|^2\,dx.
\end{align*}
As $i\ri \infty$, the first term on the right hand side goes to $0$.
It suffices to prove
\beq\label{ine}
\limsup\limits_{i\ri \infty} \int_{B_r^+}f_e(\na Q_i(x))-f_e(\na \tl{Q}(x))\,dx\leq 0.
\eeq
We shall use the minimality of $Q_i$ to show this.
Take $R\in (r,1)$ and let $\xi(x)$ be a smooth cut-off function such that
\beq
0\leq \xi\leq 1, \quad \xi\equiv 0 \text{ in }\BR^3\backslash B_R,\quad \xi\equiv 1\text{ in }B_r, \quad\mbox{ and } |\na \xi|\leq \f{2}{R-r}.
\eeq
Then we define $\{\tl{Q}^i\}$ as truncations of $\tl{Q}$ at magnitude $\f{1}{\sqrt{\e_i}}$:
\beqo
\tl{Q}^i(x)=\begin{cases}
\tl{Q}(x), & \mbox{ if }|\tl{Q}(x)|\leq \f{1}{\sqrt{\e_i}},\\
\tl{Q}(x)\f{1}{\sqrt{\e_i}|\tl{Q}(x)|}, &\mbox{ if } |\tl{Q}(x)|>\f{1}{\sqrt{\e_i}}.
\end{cases}
\eeqo
It is straightforward to verify that $\tl{Q}^i\ri \tl{Q}$ strongly in $H^1(B_1^+)$.
Define
$$
P_i=\xi \tl{Q}^i+(1-\xi) Q_i.
$$
However, one can not use $P_i$ as a comparison in \eqref{eqn: rescaled local functional}, since $P_i$ and $Q_i$ do not agree on $\G_R$.
We thus need the following technical lemma to make a correction.
\end{proof}

\begin{lemma}\label{lemma: construction of F_i}
Given $Q_i|_{\G_{1}}\ri Q|_{\G_1}$ in $C^2(\G_1)$, for any $R<1$, there exists a sequence of functions $\{F_i\}\subset H^1(B_R^+)$ such that
\begin{itemize}
  \item $F_i=\xi(Q_i-\tilde{Q})$ on $\Gamma_R$, $\mathrm{supp}\, F_i\subset B_R^+$, and $F_i\ri 0$ in $H^1(B_R^+)$;
  \item $\|F_i\|_{L^\infty}\leq C$, where $C$ is independent of $i$;
  \item In addition,
  \beq\label{eqn: convergence of bulk energy under boundary correction}
  \lim_{i\ri\infty}\int_{B_R^+}|f_b(\e_i(Q_i+F_i)+Q_{r_i})-f_b(\e_iQ_i+Q_{r_i})| \,dx=0.
  \eeq
\end{itemize}

\begin{proof}
Since $\tilde{Q}\in C^2(\Gamma_1)$, we may assume $i$ is sufficiently large such that $\tilde{Q}^i|_{\G_1} = \tilde{Q}|_{\G_1}$.
Since $Q_i|_{\G_1}\ri \tilde{Q}|_{\G_1}$ in $C^2(\G_1)$, we can easily find $\tf_i$ (e.g., by making a constant extension in the $x_3$-direction and making a suitable smooth cutoff) such that
\beqo
\tf_i=\xi(Q_i-\tilde{Q})\text{ on }\G_R,\quad \mathrm{supp}\,\tf_i\subset \overline{B_R^+},\quad  \mbox{ and }\tf_i\ri 0 \text{ in }C^2(B_R^+).
\eeqo
By the assumption on the boundary data $Q_0$ on $\Gamma_1$, there exists a universal constant $\eta>0$ such that $d(Q(r_i x))>\eta$ for all $x\in \G_1$ and all $i$.

Define a piecewise-linear function $\phi(x): [0,\infty]\ri[0,1]$ as
\beqo
\phi(x)=\begin{cases}
1, &\mbox{ if }x\geq \eta;\\
\f{x-\eta/2}{\eta/2},& \mbox{ if }\eta/2\leq x<\eta;\\
0,&\mbox{ if }0 \leq x < \eta/2.
\end{cases}
\eeqo
Then we claim that $F_i(x):=\tf_i\cdot \phi(d(Q(r_ix)))$ has the desired properties.

Firstly, thanks to the property of boundary data $Q_0$, $F_i|_{\partial B_R^+} = \tilde{F}_i|_{\partial B_R^+}$.
It then suffices to verify that $F_i\ri 0$ in $H^1(B_R^+)$, and \eqref{eqn: convergence of bulk energy under boundary correction}.

For the $H^1$-convergence, we simply calculate that
\begin{align*}
\int_{B_R^+} |\na F_i|^2\,dx \leq & \;C\int_{B_R^+}\left[ |\na \tf_i|^2\|\phi\|_{L^\infty}^2+ |\na \phi(d(Q(r_ix)))|^2\|\tf_i\|_{L^\infty}^2\right]dx\\
\leq & \;C\int_{B_R^+}\left[|\na \tf_i|^2 +\f{1}{\eta^2}|\na Q(r_ix)|^2 \|\tilde{F}_i\|_{L^\infty}^2\right]dx\\
\leq & \;C\int_{B_R^+}|\na \tf_i|^2\,dx +C\eta^{-2}\|\tilde{F}_i\|_{L^\infty}^2\cdot A_{r_i}\\
\ri &\;0 \quad \text{ as }i\ri \infty,
\end{align*}
where we used the property $|\na d(Q)|\leq C|\na Q|$ in the second inequality (see Lemma \ref{lemma: d(Q) is Lipschitz}).

To show \eqref{eqn: convergence of bulk energy under boundary correction}, we note that $F_i = 0$ if $d(Q(r_i x)) = d(\e_i Q_i(x)+Q_{r_i})<\eta/2$.
On the other hand, by assuming $i$ to be sufficiently large and using the fact that $F_i$ are uniformly bounded, we have
$$
d(Q(r_i x)+\e_i F_i)\geq \eta/4\quad \mbox{ if } d(Q(r_i x))\geq \eta/2.
$$
Therefore, $f_b(\e_i(Q_i+F_i)+Q_{r_i})<+\infty$ when $f_b(\e_iQ_i+Q_{r_i})<+\infty$, and
\begin{align*}
&\;\int_{B_R^+}|f_b(\e_i(Q_i+F_i)+Q_{r_i})-f_b(\e_iQ_i+Q_{r_i})| \,dx\\
=&\;\int_{B_R^+\cap \{d(Q(r_ix))\geq \eta/2\}}|f_b(Q(r_i x)+\e_i F_i)-f_b(Q(r_ix))| \,dx\\
\leq &\;\int_{B_R^+\cap \{d(Q(r_ix))\geq \eta/2\}}|\e_i F_i| \sup_{\{Q:\,d(Q)\geq \eta/4\}}|Df_b|\,dx\\
\leq &\;C\e_i.
\end{align*}
This completes the proof.
\end{proof}
\end{lemma}

\begin{proof}[Proof of Lemma \ref{H1strong} (continued)]
To this end, we define
$$
G_i:= P_i+F_i = \xi \tl{Q}^i+(1-\xi)Q_i+F_i
$$
as a comparison in \eqref{eqn: rescaled local functional}.
Indeed, when $i$ is sufficiently large, it can be shown that $\e_i Q_i(x)+Q_{r_i} +\e_iF_i(x)= Q(r_i x)+\e_i F_i(x)\in \BQ_{phy}$, and $\e_i \tilde{Q}^i+Q_{r_i}+\e_iF_i\in \BQ_{phy}$ since $Q_{r_i}\to Q_0(0)$ by Lemma \ref{Lebesguepoint}.
This implies that $\e_i G_i+Q_{r_i}\in \BQ_{phy}$.
Thanks to the assumptions on $F_i$, it is also easy to verify that $G_i|_{\partial B_R^+} = Q_i|_{\partial B_R^+}$.
Hence, by the (local) minimality of $Q_i$,
\begin{align*}
&\;\int_{B_R^+}\left[f_e\left(\f{\e_i}{r_i}J(r_ix)\na Q_i\right)+f_b(\e_iQ_i+Q_{r_i})\right]g(r_i x)\,dx\\
&\;\leq \int_{B_R^+}\left[f_e\left(\f{\e_i}{r_i}J(r_ix)\na G_i\right)+f_b(\e_iG_i+Q_{r_i})\right]g(r_i x)\,dx.
\end{align*}
Since $f_e$ is quadratic and $P_i= \tilde{Q}^i$ on $B_r^+$, we may rewrite this inequality as
\beqo
\begin{split}
&\;\int_{B_r^+}\left[f_e(J(r_i x)\na Q_i)-f_e(J(r_i x)\na \tl{Q})\right]g(r_ix)\,dx\\
\leq &\;
\int_{B_R^+}\left[ f_e(J(r_i x)\na (P_i+F_i))-f_e(J(r_i x)\na P_i)\right]g(r_i x)\,dx\\
&\;+\int_{B_r^+}\left[f_e(J(r_ix)\na \tl{Q}^i)-f_e(J(r_ix)\na \tl{Q})\right]g(r_ix)\,dx\\
&\;+\int_{B_R^+\backslash B_r^+}[ f_e(J(r_ix)\na P_i)-f_e(J(r_ix)\na Q_i)]g(r_ix)\,dx\\
&\;+\f{r_i^2}{\e_i^2}\int_{B_R^+} [ f_b(\e_iG_i+Q_{r_i})-f_b(\e_iQ_i+Q_{r_i})]g(r_ix)\,dx\\
\triangleq &\;I_1+I_2+I_3+\e_i^{-2}r_i^2I_4.
\end{split}
\eeqo
Since $F_i\ri 0$ in $H^1(B_1^+)$ and $\tl{Q}^i\ri \tl{Q}$ in $H^1(B_1^+)$, we use \eqref{eqn: assumption on J} and \eqref{eqn: assumption on g} to derive that $I_1+I_2\ri 0$ as $i\ri\infty$.
For $I_3$, we calculate by \eqref{eqn: assumption on J} and \eqref{eqn: assumption on g} again that
\begin{align*}
I_3\leq &\;C\int_{B_R^+\backslash B_r^+ }(|\na P_i|+|\na Q_i|)|\na (P_i- Q_i)|\,dx\\
\leq &\;C\int_{B_R^+\backslash B_r^+ }\left[\f{1}{R-r}|\tl{Q}^i-Q_i||\na Q_i|+\f{1}{(R-r)^2}|\tl{Q}^i-Q_i|^2\right]dx+C\mu_i(B_R^+\backslash B_r^+).
\end{align*}
Hence,
$$
\limsup_{i \ri \infty}I_3 \leq C\mu(B_R^+\backslash B_r^+)\quad \text{as }i\ri \infty.
$$
Here we used $L^2$-convergence of $Q_i$ and $H^1$-convergence of $\tl{Q}^i$.
For $I_4$, by the convexity of $f_b$,
\begin{align*}
I_4\leq & \; \int_{B_R^+}(1-\xi)[f_b(\e_i(Q_i+F_i)+Q_{r_i})-f_b(\e_iQ_i+Q_{r_i})] g(r_ix) \,dx\\
&\;+\int_{B_R^+}\xi[ f_b(\e_i(\tl{Q}^i+F_i)+Q_{r_i})-f_b(\e_iQ_i+Q_{r_i})]g(r_ix) \,dx\\
=:&\; I_{4,1}+I_{4,2}.
\end{align*}
By the construction of $F_i$ and the boundedness of $g(r_ix)$, $I_{4,1}\ri 0$ as $i\ri \infty$.
For $I_{4,2}$, noting that $|\e_i(\tl{Q}^i+F_i)|\leq C\sqrt{\e_i}$ and $Q_{r_i}\ri Q_0(0)$ by Lemma \ref{Lebesguepoint}, when $i$ is sufficiently large, for all $x\in B_R^+$,
$$
d(\e_i(\tl{Q}^i+F_i)+Q_{r_i})\geq \frac{1}{2}d(Q_0(0))>0.
$$
Hence, we may take $d_*\in (0, d(Q_0(0))/2)$, such that if $d(\e_i Q_i+Q_{r_i})<d_*$,
$$
f_b(\e_i(\tl{Q}^i+F_i)+Q_{r_i})\leq f_b(\e_iQ_i+Q_{r_i}).
$$
Hence, for $i$ sufficiently large, with $g(r_ix)>0$,
\begin{align*}
I_{4,2}\leq &\;\int_{B_R^+\cap\{d(\e_iQ_i+Q_{r_i})\geq d_*\}}\xi[ f_b(\e_i(\tl{Q}^i+F_i)+Q_{r_i})-f_b(\e_iQ_i+Q_{r_i})]g(r_ix)\,dx\\
\leq &\;C\int_{B_R^+\cap\{d(\e_iQ_i+Q_{r_i})\geq d_*\}}\sup_{d(Q)\geq d_*}|Df_b(Q)|\cdot \e_i|\tl{Q}^i+F_i-Q_i|\,dx.
\end{align*}
Then $I_{4,2}\ri 0$ as $i\ri \infty$ by $L^2$-convergence of $\tilde{Q}_i$, $F_i$ and $Q_i$.

Therefore,
\beqo
\limsup_{i\ri \infty} \int_{B_r^+}\left[f_e(J(r_ix)\na Q_i)-f_e(J(r_ix)\na \tl{Q})\right]g(r_ix)dx\leq \mu(B_R^+\backslash B_r^+),
\eeqo
Sending $R\ri r_+$, we prove that
\beqo
\limsup_{i\ri \infty} \int_{B_r^+}\left[f_e(J(r_ix)\na Q_i)-f_e(J(r_ix)\na \tl{Q})\right]g(r_ix)dx\leq 0.
\eeqo
Then \eqref{ine} immediately follows from \eqref{eqn: assumption on J} and \eqref{eqn: assumption on g}.
\end{proof}

\begin{lemma}\label{minlemmma}
For $\forall\, r\in(0,1)$, $\tl{Q}$ is the minimizer of
\beqo
W[Q]=\int_{B_r^+} |\na Q|^2+A|\mathrm{div}\, Q|^2\,dx.
\eeqo
\end{lemma}
\begin{proof}
For any $\phi \in C_c^\infty(B_r^+, \BQ)$, we take
$$
H_i:= \xi (\tl{Q}^i+\phi) +(1-\xi)Q_i+F_i
$$
as a comparison in \eqref{eqn: rescaled local functional}.
Here $\tl{Q}^i$, $Q_i$ and $F_i$ are defined as before.
Arguing as in the analysis of $I_4$ in the proof of Lemma \ref{H1strong}, we find that as $i\ri\infty$,
\beqo
\f{r_i^2}{\e_i^2}\int_{B_R^+} [ f_b(\e_i H_i+Q_{r_i})-f_b(\e_iQ_i+Q_{r_i})]g(r_ix)\,dx\ri 0.
\eeqo
By the minimality of $Q_i$,
\beqo
\f{r_i^2}{\e_i^2}\int_{B_r^+}f_e\left(\f{\e_i}{r_i}J(r_ix)\na Q_i(x)\right)g(r_ix)\,dx \leq \f{r_i^2}{\e_i^2}\int_{B_r^+}f_e\left(\f{\e_i}{r_i}J(r_ix)\na H_i(x)\right)g(r_ix)\,dx+o(1).
\eeqo
Since $Q_i,\tl{Q}^i\ri \tl{Q}$ strongly in $H^1(B_r^+)$, letting $i\ri \infty$ yields that
\beq\label{min3}
\int_{B_r^+}f_e(\na \tl{Q})\,dx\leq \int_{B_r^+}f_e(\na(\tl{Q}+\phi))\,dx.
\eeq
Here we used the fact that $\phi$ is supported on $B_r^+$ while $\xi \equiv 1$ on $B_r^+$.
By approximation we observe that \eqref{min3} still holds for $\phi\in H_0^1(B_r^+,\BQ)$ and this completes the proof.
\end{proof}

\begin{proof}[Proof of Lemma \ref{smallene}]
Thanks to Lemma \ref{minlemmma}, $\tl{Q}$ satisfies the following Euler-Lagrange equation in $B_1^+$:
\beqo
-\Delta \tl{Q}-A\na \mathrm{div}\tl{Q}=0,
\eeqo
subject to smooth boundary data on $\Gamma_1$.
By the elliptic regularity theory, $\tilde{Q}$ is smooth in $\overline{B_{1/2}^+}$, and there exists $\theta_* \in (0,1/4)$, such that for $\forall\,x\in \overline{\Gamma_{1/4}}$,
\beq
A_{x,\theta_*}(\tl{Q})\leq \f13 A_{1}(\tl{Q}).
\eeq

To this end, we go back to the argument that follows the statement of Lemma \ref{smallene}.
We take $\theta$ there to be $\theta_*$.
Since $\na Q_i\ri \na \tl{Q}$ in $L^2_{loc}(B_1^+)$, we find that $A_{x_*,\theta}(\tl{Q})\geq \f{1}{2}A_{1}(\tl{Q})$, which is a contradiction.
\end{proof}

The proof of Proposition \ref{pabdy} is then straightforward.
\begin{proof}[Proof of Proposition \ref{pabdy}]
We take $x_0 = 0$ as before.
First let $\e$ be defined as in Lemma \ref{smallene}.
We may assume, by decreasing $R$ if needed, that
\beq\label{eqn: small deformation}
\|\nabla (\psi^{-1})-Id\|_{L^\infty(B_{R}^+)} \ll 1.
\eeq
Then take $R\leq \e$ such that $A_R \leq \e^2$.
Consider an arbitrary $x\in B_{R/4}^+\cup \G_{R/4}$.
We shall show that for some $\alpha\in (0,1)$ and $C>0$,
\beq\label{eqn: decay of normalized dirichlet energy}
\frac{1}{r}\int_{B_r(x)\cap B_{R}^+} |\nabla Q|^2 \,dy \leq C\e^2\left(\frac{r}{R}\right)^\alpha\quad \mbox{ for all }r\in (0,R/4).
\eeq

\setcounter{case}{0}
\begin{case}\label{case: x is on the boundary}
When $x\in \Gamma_{R/4}$, it is readily proved by Lemma \ref{smallene} that $A_{x,\theta R}<A_R/2\leq \e^2/2$.
Then we may repeatedly apply Lemma \ref{smallene} with fixed base $x$ to find that $A_{x,\theta^k R}\leq \e^2/2^k$ for all $k\in \mathbb{Z}_+$, which implies \eqref{eqn: decay of normalized dirichlet energy}.

\end{case}

Next we consider $x\in B_{R/4}^+$. % and $r<R/4$.
Let $x'\in \G_{R/4}$ be the orthogonal projection of $x$ onto $\Gamma_{R/4}$.
Denote $d_x = d(x,\G_{R}) = |x-x'|$; here $d(\cdot, \cdot)$ denotes the usual Euclidean distance.

\begin{case}\label{case: radius is of the same order as distance to boundary}
If $r\geq d_x/2$, then $B_r(x)\cap B_R^+\subset B_{3r}(x')\cap B_R^+$.

If $3r\leq R/4$, by the discussion in Case \ref{case: x is on the boundary},
$$
\frac{1}{r}\int_{B_r(x)\cap B_R^+}|\nabla Q|^2\,dy\leq \frac{3}{3r}\int_{B_{3r}(x')\cap B_R^+}|\nabla Q|^2\,dy\leq C\e^2\left(\frac{r}{R}\right)^\alpha.
$$
Otherwise, $r\in ( R/12, R/4)$ and \eqref{eqn: decay of normalized dirichlet energy} is trivially true.
\end{case}

\begin{case}
Now let us assume $r< d_x/2$.
Consider $\psi(B_r(x))\subset\subset \psi(B_{d_x}(x))\subset\subset \Omega$ in the original coordinate before flattening the boundary.
By assumption \eqref{eqn: small deformation}, the deformation induced by $\psi$ is so small that we may assume that
$$
d(\psi(B_r(x)), \partial (\psi(B_{d_x}(x))) \geq C(d_x-r)\geq Cd_x.
$$
Hence, the interior $H^2$-regularity established in Proposition \ref{prop: higher interior regularity}, as well as its proof, implies that
\beqo
\begin{split}
\|\nabla Q \|_{L^2(B_r(x))}
\leq &\;C\|\nabla (Q\circ \psi^{-1}) \|_{L^2(\psi(B_r(x)))}\\
\leq &\;C|\psi(B_r(x)))|^{1/3}\|\nabla (Q\circ \psi^{-1}) \|_{L^6(\psi(B_r(x)))}\\
\leq &\;Cr\cdot d_x^{-1}\|\nabla (Q\circ \psi^{-1}) \|_{L^2(\psi(B_{d_x}(x)))}\\
\leq &\;Cr\cdot d_x^{-1}\|\nabla Q \|_{L^2(B_{d_x}(x))}.
\end{split}
\eeqo
Note that with abuse of notations here, $Q$ denotes the minimizer of \eqref{bdyfunc} with boundary flattened, while $Q\circ \psi^{-1}$ is the minimizer in the original coordinate.
Therefore,
$$
\frac{1}{r}\int_{B_r(x)}|\nabla Q|^2\,dy\leq \frac{Cr}{d_x}\cdot \frac{1}{d_x} \int_{B_{d_x}(x)}|\nabla Q|^2\,dy\leq \frac{Cr}{d_x}\cdot \e^2 \left(\frac{d_x}{R}\right)^\alpha\leq C\e^2 \left(\frac{r}{R}\right)^\alpha.
$$
In the second inequality, we used the estimate from Case \ref{case: radius is of the same order as distance to boundary}.
\end{case}

This completes the proof of \eqref{eqn: decay of normalized dirichlet energy}, which implies that $Q\in C^\alpha(B_{R/4}^+\cup \G_{R/4})$.

\end{proof}

\appendix
\section{Proof of \eqref{keyinequality}}\label{section: proof of key inequality}
In order to prove \eqref{keyinequality}, we first prove the following lemma.
\begin{lemma}
Let $M,\,N,\,P$ be $3\times 3$-symmetric traceless matrices.
Then
\beq\label{ineq1}
\|M\|_2^2\leq \f23|M|^2,
\eeq
and
\beq\label{ineq2}
\sum\limits_{i=1}^3(M_{i1}+N_{i2}+P_{i3})^2\leq \f53(|M|^2+|N|^2+|P|^2).
\eeq
\end{lemma}
\begin{proof}
To prove \eqref{ineq1}, we can only consider the case when $M$ is diagonal, due to the fact that $M$ is symmetric and both Frobenius norm and matrix 2-norm are unitarily invariant. Without loss of generality we assume  $M=\mathrm{diag}\{\lam_1,\lam_2,-\lam_1-\lam_2\}$ with $\lam_1\lam_2\geq 0$, we deduce that
\beqo
2|M|^2-3\|M\|_2^2
=2(\lam_1^2+\lam_2^2+(\lam_1+\lam_2)^2)-3(\lam_1+\lam_2)^2=2(\lam_1-\lam_2)^2\geq 0.
\eeqo
%Therefore we have $\|M\|_2^2\leq \f23|M|^2$. And t
Then \eqref{ineq1} follows.
The equality holds if and only if $M$ has two equal eigenvalues.

For \eqref{ineq2}, we compute
\beq
\begin{split} &\;5(|M|^2+|N|^2+|P|^2)-3\sum\limits_{i=1}^3(M_{i1}+N_{i2}+P_{i3})^2\\
\geq &\; \left(  5(M_{11}^2+M_{22}^2+M_{33}^2+2N_{12}^2+2P_{13}^2)-3(M_{11}+N_{12}+P_{13})^2  \right)\\
&\;+\left(  5(N_{11}^2+N_{22}^2+N_{33}^2+2M_{21}^2+2P_{23}^2)-3(M_{21}+N_{22}+P_{23})^2  \right)\\
&\;+\left(  5(P_{11}^2+P_{22}^2+P_{33}^2+2M_{31}^2+2N_{32}^2)-3(M_{31}+N_{32}+P_{33})^2  \right).
\end{split}
\label{ineq3}
\eeq
The equality holds if and only if $M_{23}=N_{13}=P_{12}=0$.
It suffices to prove non-negativity of each term on the right hand side of \eqref{ineq3}.
We only show this for the first term; the others can be handled similarly.
\begin{align*}
&\;5(M_{11}^2+M_{22}^2+M_{33}^2+2N_{12}^2+2P_{13}^2)-3(M_{11}+N_{12}+P_{13})^2\\
\geq &\; \left(5+\f52-3\right) M_{11}^2+7N_{12}^2+7P_{13}^2-6M_{11}N_{12}-6M_{11}P_{13}-6N_{12}P_{13}\\
=&\; \left(\f32M_{11}-2N_{12}\right)^2+\left(\f32 M_{11}-2P_{13}\right)^2+3(N_{12}-P_{13})^2\geq 0.
\end{align*}
Here we used $\sum_{i=1}^3 M_{ii}=0$ in the second line.
The equality holds if and only if
\beqo
N_{12}=P_{13}=\f34M_{11}, \quad M_{22}=M_{33}=-\f12 M_{11}.
\eeqo
This completes the proof of \eqref{ineq2}.
\end{proof}

To this end, \eqref{keyinequality} follows immediately from the lemma if we take $M=D_k^h Q$ in \eqref{ineq1}, and take $M=D_k^h \pa_1 Q,\, N=D_k^h \pa_2 Q,\, P=D_k^h \pa_3 Q$ in \eqref{ineq3}.

\section{Formula of $p(A)$}
\label{section: formula for p(A)}
\begin{lem}\label{lemma: formula for p(A)}
For $A>-\f35$, let
$$
p(A) := \sup_{\om\in [0,1],\,\om+\f{5}{3}A\geq 0}p(A,\om),
$$
where $p(A,\om)$ is defined in \eqref{eqn: constraint on p with omega}.
Then $p(A)$ is given by \eqref{eqn: def of p(A)}.
\begin{proof}
We rewrite \eqref{eqn: constraint on p with omega} as
\beq
p(A,\om) = 1+\f{\f95\left(\om+\f53 A\right)+\sqrt{\f{9}{5}\left(\om+\f53 A\right) \left[\left(
\f{9}{5}-2A^2\right)\left(\om+\f53 A\right)+2A^2\left(1+\f{5}{3}A\right)\right]}}{2A^2}.
\label{eqn: new form of p A omega}
\eeq
It is easy to see that if $\f{9}{5}-2A^2 \geq 0$, $p(A,\om)$ achieves its supremum at $\om = 1$, which gives
$$
p(A) = 1+\f{3}{A}+\f{9}{5A^2}.
$$
It suffices to consider $2A^2 > \f{9}{5}$, i.e., $A> \f{3\sqrt{10}}{10}$.
Define
\beqo
g(y) := \f{9}{5}y+\sqrt{\f{9}{5}y(B_1y+B_2)},
\eeqo
where
$$
B_1 = \f95-2A^2,\quad B_2 = 2A^2\left(1+\f53A\right).
$$
Then
$$
p(A,\om) = 1+(2A^2)^{-1}g\left(\om+\f53 A\right).
$$
Since
$$
g'(y) = \f95+\sqrt{\f95}\cdot \f{2B_1 y+ B_2}{2\sqrt{y(B_1y+B_2)}},
$$
we find that $g'(y)<0$ if and only if
$$
B_1 y +\f{B_2}{2}< -\sqrt{\f95 y(B_1y+B_2)},
$$
which is equivalent to
$$
\left(B_1 y +\f{B_2}{2}\right)^2>\f95 y(B_1y+B_2)\quad \mbox{ and }\quad B_1y +\f{B_2}{2}<0.
$$
Solving these inequalities under the assumption $A>\f{3\sqrt{10}}{10}$, we find that
$$
y>\f{A\left(1+\f53A\right)}{2A- \sqrt{\f{18}{5}}}=:y_{+}.
$$
This implies that within the domain of $g(y)$, i.e.,
$$
y = \om+\f{5}{3}A \in \left[\f53 A, 1+\f53A\right].
$$
$g(y)$ is decreasing if and only if $y\geq y_+$.
\begin{enumerate}
  \item
When
$$
1+\f53 A \leq   y_+\quad\Leftrightarrow \quad A\in \left(\f{3\sqrt{10}}{10},\sqrt{\f{18}{5}}\right],
$$
$p(A,\cdot)$ is increasing on $[0,1]$.
Hence,
$$
p(A) = p(A,1) = 1+\f{3}{A}+\f{9}{5A^2}.
$$
\item
When
$$
\f53 A <  y_+< 1+\f53 A\quad \Leftrightarrow\quad A\in \left[\sqrt{\f{18}{5}}, \f35+\sqrt{\f{18}{5}}\right],
$$
then supremum of $p(A,\cdot)$ is achieved at $\om_*$ such that
$\om_*+\f53 A = y_+$.
Combining this with \eqref{eqn: new form of p A omega} yields that
\beqo
p(A) = 1+\f{3+5A}{2\sqrt{10}A- 6}.
\eeqo

\item
When
$$
\f53 A \geq   y_+\quad\Leftrightarrow\quad A\geq \f35+\sqrt{\f{18}{5}},
$$
$p(A,\cdot)$ is decreasing on $[0,1]$.
Hence,
$$
p(A) = p(A,0) = 1+\f{3+\sqrt{9+6A}}{2A}.
$$
\end{enumerate}
This completes the derivation.
\end{proof}
\end{lem}
\section{Study of $d(Q)$}
\label{section: properties of d(Q)}
We study the properties of $d(Q)$ in this section.
It is known that every $Q\in \BQ_{phy}$ can be represented by
\begin{equation}\label{qtensor}
Q=\lambda_1 n\otimes n+\lambda_2 m\otimes m+\lam_3 p\otimes p,
\end{equation}
where
\begin{equation}\label{eqn: order of eigenvalues}
\lam_1+\lam_2+\lam_3=0,\, \lam_i\in \left[-\f13,\f23\right],\,\lam_1\leq \lam_2\leq\lam_3,
\end{equation}
and where $(n,m,p)$ forms an orthonormal frame in $\BR^3$.
Then we have the following characterization of $d(Q)$.
\begin{lemma}\label{lemma: characterization of d(Q)}
Let $Q\in \BQ_{phy}$ be given by \eqref{qtensor} and \eqref{eqn: order of eigenvalues}.
Then $d(Q)=|Q-Q'|$, where
\beq\label{dq}
Q'=-\f13 n\otimes n+\left(\lam_2+\f{\lam_1+\f13}{2}\right)m\otimes m+\left(\lam_3+\f{\lam_1+\f13}{2}\right)p\otimes p.
\eeq
As a result,
\begin{equation}\label{formula for d(Q)}
d(Q) = \frac{\sqrt{6}}{2}\left(\lambda_1+\frac{1}{3}\right).
\end{equation}

\begin{proof}
Since the distance between two matrices is invariant under orthogonal transforms, without loss of generality, we may assume $n = (1,0,0)$, $m = (0,1,0)$ and $p = (0,0,1)$.
Let $s=2\lam_1+\lam_2$ and $r=2\lam_2+\lam_1$.
Then \eqref{qtensor} becomes
\beq\label{qtensor2}
Q=s\left(n\otimes n-\f{1}{3}I\right)+r\left(m\otimes m-\f13 I\right), \quad \f{r-1}{2}\leq s\leq r\leq 0.
\eeq
Now we are going to look for $Q'\in \pq$ such that $|Q-Q'|$ is minimized.
Assume
\beqo
Q'=s'\left(n'\otimes n'-\f{1}{3}I\right)+r'\left(m'\otimes m'-\f13 I\right),
\eeqo
with
\beqo
n' = (a,b,c),\quad m'=(u,v,w),\quad \f{r'-1}{2}\leq s'\leq r'\leq 0,
\eeqo
and $(n',m')$ being an orthonormal pair.

First we are going to show that when $s',r'$ is fixed, $|Q-Q'|$ is minimized when $n'=n,\,m'=m$.
We calculate that
\begin{equation*}
\begin{split}
&\;|Q-Q'|^2\\
=&\;|Q|^2+|Q'|^2\\
&-2\left[s\left(n\otimes n-\f{1}{3}I\right)+r\left(m\otimes m-\f13 I\right)\right]:\left[s'\left(n'\otimes n'-\f{1}{3}I\right)+r'\left(m'\otimes m'-\f13 I\right)\right]\\
=&\;C(s,r,s',r')-2\left[ss'\left(( n,n')^2-\f13\right)+sr'\left(( n,m')^2-\f13\right)\right.\\
&\left.\qquad+rs'\left((m,n')^2-\f13\right)+rr'\left(( m,m')^2-\f13\right)\right]\\
= &\; C_(s,r,s',r')-2(ss'a^2+rs'b^2+sr'u^2+rr'v^2).
\end{split}
\end{equation*}
Here $C(s,r,s',r')$ represents some constant depending only on $s$, $r$, $s'$ and $r'$, whose definition changes from line to line.

Then it suffices to show that
\beq\label{Gineq}
ss'a^2+rs'b^2+sr'u^2+rr'v^2\leq ss'+rr'.
\eeq
Recall that
\begin{equation}
\label{abc} a^2+b^2+c^2=1,\quad u^2+v^2+w^2=1,\quad au+bv+cw=0.
\end{equation}
%Then $G=$. Next w
We claim that $u^2\leq b^2+c^2$. %It simply follows from the relation \eqref{abc} and Cauchy inequality:
Indeed, by \eqref{abc},
\beqo
a^2u^2=|bv+cw|^2\leq (b^2+c^2)(v^2+w^2)=(1-a^2)(1-u^2),
\eeqo
which implies that $u^2\leq 1-a^2=b^2+c^2$.
Then we deduce that
\begin{equation*}
\begin{split}
&\;ss'+rr'-(ss'a^2+rs'b^2+sr'u^2+rr'v^2)\\
=& \;(s-r)s' (b^2+c^2)-(s-r)r'u^2+rr'w^2+rs'c^2\\
\geq &\;(s-r)(s'-r')u^2\geq 0.
\end{split}
\end{equation*}
Here we used \eqref{abc} and the facts that $s\leq r\leq 0$ and $s'\leq r' \leq 0$.

To this end, we have showed that if $Q$ is given by \eqref{qtensor} and if $Q'\in \pq$ minimizes $|Q-Q'|$, $Q'$ should be represented by
$$
Q'=\mu_1 n\otimes n+\mu_2 m\otimes m+\mu_3 p\otimes p,
$$
for some $-1/3= \mu_1\leq \mu_2\leq \mu_3\leq 2/3$ such that $\mu_1+\mu_2+\mu_3 =0$.
The constraints on $\mu_i$ are due to the characterization of $\pq$ in \eqref{boundary of Q}.
Moreover,
\beq\label{distance}
|Q-Q'|=\sqrt{(\lam_1-\mu_1)^2+(\lam_2-\mu_2)^2+(\lam_3-\mu_3)^2}.
\eeq
Therefore, $|Q-Q'|$ achieves its minimum if
$$
\mu_1=-\f13,\quad\mu_2=\lam_2+\f{\lam_1+\f13}{2},\quad \mu_3=\lam_3+\f{\lam_1+\f13}{2}.
$$
\eqref{formula for d(Q)} follows immediately.
This completes the proof.
\end{proof}
\end{lemma}

An immediate consequence of Lemma \ref{lemma: characterization of d(Q)} is
\begin{lemma}\label{lemma: d(Q) is Lipschitz}
$d(Q)$ is Lipschitz continuous in $\mathcal{Q}_{phy}$.
\begin{proof}
The difference between the smallest eigenvalues of two matrice in $\mathcal{Q}_{phy}$ can be bounded by their distance.
Combining this fact with Lemma \ref{lemma: characterization of d(Q)}, we complete the proof of the Lemma.
\end{proof}
\end{lemma}

\section{A Construction of $\{f_b^\e\}$}
\label{section: construction of f_b eps}
In this section, we provide a construction of $\{f_b^\e\}_{0<\e\ll 1}$ used in Section \ref{section: proof of the improved bound}.
For convenience, we recall the conditions on $\{f_b^\e\}_{0<\e\ll 1}$:
\begin{enumerate}[($i'$)]
\item \label{assumption: non negative f b eps} For all $0<\e\ll 1$, $f_b^\e(Q)\in [0, \infty)$ for all $Q\in \mathcal{Q}$;
\item \label{assumption: smooth and convex} $f_b^\e$ are convex and smooth in $\mathcal{Q}$; %, with the Lipschitz constant depending on $\e$;
\item \label{assumption: boundedness} $f_b^\e(Q)\leq f_b(Q)$ for all $Q\in \mathcal{Q}$.
\item Moreover,
$$
\lim_{\e\ri 0^+} f_b^\e(Q) = f_b(Q),\quad \lim_{\e\ri 0^+} Df_b^\e(Q) = Df_b(Q)
$$
locally uniformly in $\BQ^{\mathrm{o}}_{phy}$.
%\item \label{assumption: gradient bound} $|Df_b^\e(Q)|^s\leq C_s'|f_b^\e(Q)|^{s+1}$ for all $Q\in\mathcal{Q}$, where $C_s'>0$ is independent of $\e$ but may depend on $s$.
\end{enumerate}
Here $Df_b^\e(Q)$ denotes the gradient of $f_b^\e$ with respect to $Q$.

\begin{proof}[Proof of Lemma \ref{lemma: construction of f_b eps}]
Define
\beqo
\BQ_{phy}^\e = \{Q\in \BQ^{\mathrm{o}}_{phy}:\, f_b(Q)< \e^{-1}\}.
\eeqo
Take $\e\ll 1$, such that $\BQ_{phy}^\e$ is a non-empty open subset of $\BQ_{phy}$.
Then we define on the entire $\BQ$ that
%Define $F^\e$ to be the convex hull of $f_b|_{\BQ_{phy}}^\e$, i.e.,
\beqo
F_b^\e(Q) = \sup_{Q'\in \BQ_{phy}^\e} f_b(Q')+Df_b(Q')(Q-Q').
\eeqo
It is not difficult to show that $\{F_b^\e\}_{0<\e\ll 1}$ satisfies all the conditions above except for the smoothness issue.
Indeed, $F_b^\e \equiv f_b$ on $\BQ_{phy}^\e$, while outside $\BQ_{phy}^\e$, $F_b^\e$ is only Lipschitz continuous and $DF_b^\e$ exists in the $L^\infty$-sense but may not be well-defined pointwise.
In particular, for all $Q_1,Q_2\in \BQ$,
\beq
|F_b^\e(Q_1)-F_b^\e(Q_2)|\leq |Q_1-Q_2|\sup_{\BQ_{phy}^\e}|Df_b| =: |Q_1-Q_2|\om_\e .
\label{eqn: Lipschitz norm of F_b eps}
\eeq
Note that $\om_\e \to +\infty$ as $\e\to 0^+.$

We shall make a little modification of $\{F_b^\e\}$ to construct smooth $\{f_b^\e\}$.
Let $\phi$ be a non-negative $C_0^\infty$-mollifier in $\mathcal{Q}$ supported on the unit ball, such that $\int_\mathcal{Q}\phi(Q)\,dQ = 1$.

Then we define
\beqo
f_b^\e(Q) = \int_{\BQ} \phi(Q')F_b^\e(Q-\e \om_\e^{-1} Q')\,dQ'-\e.
\eeqo

We derive that for arbitrary $Q\in \BQ$,
\beq
\label{eqn: error between f_b eps and F_b eps}
\begin{split}
|f_b^\e(Q)+\e-F_b^\e (Q)| \leq &\; \int_{\BQ}\phi(Q')|F_b^\e(Q-\e \om_\e^{-1} Q')-F_b^\e(Q)|\,dQ'\\
\leq &\; \int_{\BQ}\phi(Q')\cdot \e \om_\e^{-1}\cdot \om_\e\,dQ'= \e.
\end{split}
\eeq
In the first inequality, we used the fact that $\phi$ is non-negative and normalized; in the second inequality, we applied \eqref{eqn: Lipschitz norm of F_b eps} as well as that $\phi$ is supported on the unit ball in $\BQ$.
\eqref{eqn: error between f_b eps and F_b eps} implies that $F_b^\e(Q)-2\e\leq f_b^\e(Q)\leq F_b^\e(Q)\leq f_b(Q)$.

It is then easy to verify that $\{f_b^\e\}_{0<\e\ll 1}$ satisfies all the conditions we need.
\end{proof}

\par
~\\
\textbf{Acknowledgement.} This work was supported by NSF grant DMS-1501000. We want to thank Professor Fanghua Lin for introducing us to this obstacle problem, and for his valuable comments and suggestions in the preparation of this work.

\bibliographystyle{plain}
\bibliography{Regularity-singular}

\end{document}